\documentclass[12pt, a4paper]{amsart}
\usepackage{hyperref, fullpage, amsmath, amsfonts, amsthm, amssymb, stmaryrd}
\usepackage{color}
\usepackage[all]{xy}
\usepackage{comment}

\theoremstyle{plain}
\newtheorem{Theorem}{Theorem}[section]
\newtheorem{Lemma}[Theorem]{Lemma}
\newtheorem{Corollary}[Theorem]{Corollary}
\newtheorem{Proposition}[Theorem]{Proposition}

\theoremstyle{remark}
\newtheorem{Remark}[Theorem]{Remark}
\newtheorem{Example}[Theorem]{Example}
\newtheorem{Definition}[Theorem]{Definition}

\newtheorem{Question}[Theorem]{Question}

\address{Jeremiah Horrocks Institute, University of Central Lancashire, Preston PR1 2HE, United Kingdom}
\email{sanscombe@uclan.ac.uk}
\address{Fachbereich Mathematik und Statistik, University of Konstanz, 78457 Konstanz, Germany}
\email{arno.fehm@uni-konstanz.de}

\begin{document}

\title[Characterizing diophantine henselian valuation rings and valuation ideals]{Characterizing diophantine henselian\\ valuation rings and valuation ideals}
\author{Sylvy Anscombe and Arno Fehm}
\thanks{\today\\During this research the first author was funded by EPSRC grant EP/K020692/1.}

\begin{abstract}
We give a characterization, in terms of the residue field, of those henselian valuation rings 
and those henselian valuation ideals that are diophantine.
This characterization gives a common generalization of all the positive and negative results  
on diophantine henselian valuation rings and diophantine valuation ideals in the literature.
We also treat questions of uniformity and we apply the results to show that a given field can carry at most one
diophantine nontrivial equicharacteristic henselian valuation ring or valuation ideal.
\end{abstract}

\maketitle

\section{Introduction}

\noindent
We study diophantine subsets of a field $K$,
that is, sets $X\subseteq K$ that are the projection of the common zero set $Z\subseteq K^n$ of finitely many polynomials 
$$
 f_1,\dots,f_r\in\mathbb{Z}[x_1,\dots,x_n]
$$
onto the first coordinate.
Diophantine subsets of fields have been studied for a long time in various areas, like number theory, arithmetic geometry and mathematical logic, 
and a wide variety of methods has been developed.
(From a model theoretic point of view, a subset of $K$ is diophantine if it is definable by an existential parameter-free formula
in the language $\mathcal{L}_{\rm ring}$ of rings.)
See for example \cite{Denef, Shlapentokh, Kollar, Koenigsmann}
and the references therein
for problems and results on diophantine subsets of number fields, function fields, and certain infinite algebraic extensions of $\mathbb{Q}$.

The subsets $X\subseteq K$ we are mainly concerned with are the valuation ring $X=\mathcal{O}_v$ 
and the valuation ideal $X=\mathfrak{m}_v$ of some henselian valuation $v$ on $K$.
As an example, already Julia Robinson observed in \cite{Robinson} that the valuation ring $\mathbb{Z}_p$ inside the field of $p$-adic numbers $\mathbb{Q}_p$ is diophantine:
For $p>2$, it is the projection of the zero set
$$
 Z=\left\{ (x,y)\in\mathbb{Q}_p^2 : 1+px^2-y^2=0 \right\}
$$
onto the first coordinate.
The analogous result for finite extensions of $\mathbb{Q}_p$ requires more care and was worked out much later in \cite{Cluckersetal}.
Similarly, an analogous definition for the local fields $\mathbb{F}_p((t))$ is more complicated and was obtained only recently in \cite{AnscombeKoenigsmann}.
In \cite{Fehm} it was observed that this last result does not make use of the fact that the value group of the valuation is discrete but can be extended to arbitrary henselian valuations with residue field $\mathbb{F}_p$.
Moreover, the methods were extended to certain henselian valued fields with residue field pseudo-algebraically closed (PAC) or pseudo-real closed (PRC).

Our first result is that the observation in \cite{Fehm} reflects in fact a general principle:
The question of whether a henselian valuation ring is diophantine never depends on its value group, but only on its residue field
(see below for the precise statement).
Moreover, we are able to isolate a condition that characterizes those residue fields for which the valuation ring, respectively the valuation ideal,
of the henselian valuation is diophantine:

\begin{Theorem}\label{thm:introthm}
Let $F$ be a field. Then the following are equivalent:
\begin{enumerate}
\item There is an $\exists$-$\mathcal{L}_{\mathrm{ring}}$-formula that defines $\mathcal{O}_v$ [respectively, $\mathfrak{m}_v$] in $K$
for {\em some} equicharacteristic henselian nontrivially valued field $(K,v)$ with residue field $F$.
\item There is an $\exists$-$\mathcal{L}_{\mathrm{ring}}$-formula that defines $\mathcal{O}_v$ [respectively, $\mathfrak{m}_v$] in $K$
for {\em every} henselian valued field $(K,v)$ with residue field elementarily equivalent to $F$.
\item There is no elementary extension $F\preceq F^*$ with a nontrivial valuation $v$ on $F^*$
for which the residue field $F^*v$ embeds into $F^*$
[respectively, with a nontrivial henselian valuation $v$ on a subfield $E$ of $F^*$ with $Ev\cong F^*$].
\end{enumerate}
\end{Theorem}

In fact, we get further interesting equivalent statements
and also treat 
definitions with parameters from a subfield of $F$,
as well as questions of uniform definitions for classes of fields $F$.
See Section \ref{section:main.theorem} for a summary of the main results, including a proof of \ref{thm:introthm}.

Using \ref{thm:introthm}, we can easily reprove all the definability results mentioned above (namely where the residue field
is finite, PAC or PRC), get new definability results (for example in the case where the residue field is P$p$C or equals $\mathbb{Q}$),
and also acquire several new negative results (\ref{cor:positive.applications}, \ref{cor:negative.applications}, \ref{cor:positive.large}, \ref{cor:negative.applications.large}).
Using the more general results, we for example gain new insight into the question
for which sets of prime numbers $P$ there are uniform $\exists$-$\emptyset$-definitions
for the valuation rings in the families $\{\mathbb{Q}_p:p\in P\}$ and $\{\mathbb{F}_p((t)):p\in P\}$ (\ref{cor:QpZp}).
We finally combine \ref{thm:introthm} and the negative results to conclude that
a given field admits at most one nontrivial equicharacteristic henselian valuation with
diophantine valuation ring or diophantine valuation ideal (\ref{thm:classification}).

The paper is organized as follows:
After some preliminaries in Section \ref{sec:prelim},
we start in Section \ref{sec:embedded} with the proof of $(2)\Leftrightarrow(3)$,
which is based on a simple non-constructive but powerful characterization result of Prestel. 
The proof of $(1)\Leftrightarrow(2)$, which we present in Section \ref{sec:uniform}, is more technical and builds on 
our work \cite{AnscombeFehm} on the existential theory of equicharacteristic henselian valued fields.
After putting everything together in Section \ref{section:main.theorem},
we give the above mentioned applications and a few further corollaries in Section \ref{section:examples}.

\section{Notation and preliminaries}
\label{sec:prelim}

\subsection{Valued fields}

Let $K$ be a field
and $v:K\rightarrow\Gamma\cup\{\infty\}$ a valuation on $K$.
We denote by $vK=v(K^\times)$ the value group of $v$,
by $\mathcal{O}_v=\{x\in K:v(x)\geq0\}$ the valuation ring of $v$,
by $\mathfrak{m}_v=\{x\in K:v(x)>0\}$ the valuation ideal of $v$,
and by $Kv=\mathcal{O}_v/\mathfrak{m}_v$ the residue field of $v$.
For $\gamma\in vK$ and $a\in K$ we denote by $B_K(\gamma,a):=\{x\in K:v(x-a)\geq\gamma\}$
the ball of radius $\gamma$ around $a$.
We denote by $K^{\mathrm{alg}}$ an algebraic closure of $K$ and by $\mathrm{Gal}(K)$ the absolute Galois group of $K$, i.e.~the group of automorphisms of $K^{\mathrm{alg}}$ that fix $K$ pointwise.

We will make use of the following well-known valuation theoretic facts:

\begin{Lemma}\label{lem:complete.the.square}
Let $(K,v)$ be a valued field and let $F/Kv$ be any field extension. Then there is an extension of valued fields $(L,w)/(K,v)$ such that
\begin{enumerate}
\item $L/K$ is separable,
\item $Lw/Kv$ is isomorphic to the extension $F/Kv$,
\item $w$ is nontrivial, and
\item $w$ is henselian.
\end{enumerate}
\end{Lemma}

\begin{proof}
Embed $vK$ into a nontrivial ordered abelian group $\Gamma$.
By Theorem 2.14 from \cite{Kuhlmann04a}, there is an extension $(L,w)/(K,v)$ such that $L/K$ is separable, $Lw/Kv$ is isomorphic to $F/Kv$, and $wL=\Gamma$ (thus $w$ is nontrivial).
By passing to the henselisation, which is an immediate separable extension, we may also suppose that $(L,w)$ is henselian.
\end{proof}

\begin{Remark}\label{rem:henselian.ext}
We will on several occasions have to construct a nontrivial henselian valued extension $(K,v)$ of a trivially valued field $F$
with $Kv=F$.
One way to obtain such an extension is via \autoref{lem:complete.the.square},
but we could as well take a concrete extension like $K=F((t))$ with $v=v_t$ the $t$-adic valuation.
\end{Remark}

\begin{Lemma}\label{lem:extend.with.alg.res.field}
Let $(K,v)$ be a valued field and $L/K$ a field extension.
Then there exists an extension of $v$ to a valuation $w$ on $L$ such that
$Lw/Kv$ is algebraic.
\end{Lemma}

\begin{proof}
This follows for example from the version of Chevalley's theorem in
\cite[p.~8 Thm.~1]{LangAlgebraicGeometry}.
\end{proof}

\begin{Lemma}\label{lem:residue.of.sep.closed}
Let $F$ be separably closed and let $v$ be a nontrivial valuation on $F$. Then $Fv$ is algebraically closed.
\end{Lemma}

\begin{proof}
See \cite[Theorem 3.2.11]{EP}.
\end{proof}

\subsection{The category of $C$-fields}

In this work,
we always fix a ring $C$ and work in the category of {\em $C$-fields},
i.e.\ fields $F$ together with a fixed (structure) homomorphism $C\rightarrow F$.
All fields are understood to be $C$-fields,
all embeddings of fields are understood to be $C$-embeddings, i.e.\ to commute with the structure homomorphism,
and all valuations on fields are understood to be $C$-valuations, i.e.\ nonnegative on (the image of) $C$.
The residue field of such a valuation naturally admits the structure of a $C$-field, simply by composing the structure homomorphism with the residue map, and we will tacitly view them as such.
Similarly, also any field extension of a $C$-field is naturally again a $C$-field.
For a $C$-field $F$ we denote by $C_F\subseteq F$ the quotient field of the image of the structure homomorphism $C\rightarrow F$.

An important example is the case $C=\mathbb{Z}$: In this case, any field $F$ can be turned into a $C$-field in a unique way,
and $C_F$ is the prime field of $F$.
Another important example is the case where $C$ is a field:
In this case, the $C$-fields are exactly the field extensions of $C$.

\subsection{Model theory of valued fields}

We now fix the languages in which we are going to work.
Let 
$$
 \mathcal{L}_{\mathrm{ring}}=\{+,-,\cdot,0,1\}
$$ 
be the language of rings and let 
$$
 \mathcal{L}_{\rm vf}=\{+^K,-^K,\cdot^K,0^K,1^K,+^k,-^k,\cdot^k,0^k,1^k,+^\Gamma,<^\Gamma,0^\Gamma,\infty^\Gamma,v,{\rm res}\}
$$ 
be a three sorted language for valued fields
with a sort $K$ for the field itself, a sort $\Gamma\cup\{\infty\}$ for the value group with infinity, and a sort $k$ for the residue field,
as well as both the valuation map $v$ and the residue map ${\rm res}$,
which we interpret as the constant $0^k$ map outside the valuation ring. 
For a ring $C$, we let $\mathcal{L}_{\mathrm{ring}}(C)$ and $\mathcal{L}_{\rm vf}(C)$ be the languages obtained by adding symbols for elements of $C$,
where in the case of $\mathcal{L}_{\rm vf}(C)$, the constant symbols are added to the field sort $K$.
A valued $C$-field $(K,v)$ gives rise in the usual way to an $\mathcal{L}_{\rm vf}(C)$-structure 
$$
 (K,vK\cup\{\infty\},Kv,v,\mathrm{res},c^K)_{c\in C},
$$
where $vK$ is the value group, $Kv$ is the residue field, and $\mathrm{res}$ is the residue map.
For notational simplicity, we will usually write $(K,v)$ to refer to the $\mathcal{L}_{\rm vf}(C)$-structure it induces. 
The class of $C$-fields forms an elementary class in $\mathcal{L}_{\rm ring}(C)$,
and the class of valued $C$-fields forms an elementary class in $\mathcal{L}_{\rm vf}(C)$.

We remind the reader that most model theoretic constructions, like ultraproducts, can be carried out in a three-sorted language just like for the usual one-sorted languages, see e.g.~\cite[Section 1.7]{Chatzidakis}.
In particular, for ultraproducts of valued fields one has:

\begin{Lemma}\label{lem:residue.ultraproduct}
Let $(K_i,v_i)_{i\in I}$ be a family of valued fields and $\mathcal{F}$ an ultrafilter on $I$.
Then the ultraproduct $(K,v)=\prod_{i\in I}(K_i,v_i)/\mathcal{F}$ is a valued field
with residue field
$Kv=\prod_{i\in I}K_iv_i/\mathcal{F}$.
\end{Lemma}

\subsection{Definable sets}

Let $\mathcal{L}$ be a first order language.
For an $\mathcal{L}$-formula $\phi(x)$ in one free variable $x$ 
and an $\mathcal{L}$-structure $K$
we denote by $\phi(K)=\{x\in K: K\models\phi(x)\}$ the set defined by $\phi$ in $K$.
A subset $X\subseteq K$ is {\em definable} if there exists some $\phi$ with $X=\phi(K)$,
in which case we also call $\phi$ a {\em definition} for $X$.

As usual, we say that an $\mathcal{L}$-formula is an 
{\em $\exists$-$\mathcal{L}$-formula}
(resp.~{\em $\forall$-$\mathcal{L}$-formula})
if it is logically equivalent to a formula in prenex normal form with only existential
(resp.~universal)
quantifiers.
We stress that for the three-sorted language $\mathcal{L}_{\rm vf}(C)$ the existential
(resp.~universal)
quantifiers in the above-mentioned prenex normal form
may quantify over any of the sorts.
We say that an $\mathcal{L}_{\rm vf}(C)$-sentence is an {\em $\forall^{k}\exists$-$\mathcal{L}_{\rm vf}(C)$-sentence} if it is logically equivalent to 
a sentence of the form $\forall\mathbf{x}\,\psi(\mathbf{x})$,
where $\psi$ is an $\exists$-$\mathcal{L}_{\rm vf}(C)$-formula and
the universal quantifiers range over the residue field sort.

If we say that a subset $X$ of a $C$-field $K$ is {\em $\exists$-$C$-definable}
(resp.~{\em $\forall$-$C$-definable}),
we always mean that it is definable
by an
$\exists$-$\mathcal{L}_{\mathrm{ring}}(C)$-formula
(resp.~$\forall$-$\mathcal{L}_{\mathrm{ring}}(C)$-formula);
and if we say that $X$ is {\em $\exists$-$\emptyset$-definable}
(resp.~{\em $\forall$-$\emptyset$-definable}),
we  mean that it is definable
by an
$\exists$-$\mathcal{L}_{\mathrm{ring}}$-formula
(resp.~$\forall$-$\mathcal{L}_{\mathrm{ring}}$-formula).
For a class $\mathcal{K}$ of valued $C$-fields,
we say that the valuation ring resp.~the valuation ideal is
{\em uniformly $\exists$-$C$-definable}
in $\mathcal{K}$
if there is an $\exists$-$\mathcal{L}_{\mathrm{ring}}(C)$-formula
$\phi(x)$ such that
$\mathcal{O}_{v}=\phi(K)$
(resp.~$\mathfrak{m}_v=\phi(K)$)
for each $(K,v)\in\mathcal{K}$.
The term {\em uniformly $\forall$-$C$-definable} is used analogously.

\begin{Lemma}\label{lem:EAOm}
Let $\mathcal{K}$ be a class of valued $C$-fields.
Then the valuation ring is uniformly $\exists$-$C$-definable in $\mathcal{K}$ 
if and only if the valuation ideal is uniformly $\forall$-$C$-definable in $\mathcal{K}$,
and the valuation ring is uniformly $\forall$-$C$-definable in $\mathcal{K}$ if and only if 
the valuation ideal is uniformly $\exists$-$C$-definable in $\mathcal{K}$.
\end{Lemma}

\begin{proof}
Let $(K,v)\in\mathcal{K}$.
If $\phi(x)$ defines $\mathcal{O}_v$
(resp.~$\mathfrak{m}_v$), then both
$$
 x=0\vee\exists y\;(xy=1\wedge\neg\phi(y))
$$ 
and
$$
 x=0\vee\forall y\;(xy=1\rightarrow\neg\phi(y))
$$ 
define $\mathfrak{m}_v$
(resp.~$\mathcal{O}_v$).
\end{proof}

\begin{Remark}\label{rem:EAOm}
Due to this observation we can either talk about $\exists$-$C$-definable valuation rings and valuation ideals
(which we refer to as the {\em $\mathcal{O}$-case} and the {\em $\mathfrak{m}$-case}),
or, {\em equivalently}, about $\exists$-$C$-definable and $\forall$-$C$-definable valuation rings
(which we refer to as the {\em $\exists$-case} and the {\em $\forall$-case}).
We will frequently switch between these two viewpoints.
\end{Remark}

We will use the following criterion which is a special case of \cite[Characterization Theorem]{Prestel}.

\begin{Proposition}
\label{thm:Prestel}
Let $\mathcal{K}$ be an $\mathcal{L}_{\mathrm{vf}}(C)$-elementary class of valued $C$-fields.
\begin{enumerate}
\item The valuation ring is uniformly $\exists$-$C$-definable in $\mathcal{K}$
if and only if
$$(K_{1}\subseteq K_{2}\implies\mathcal{O}_{v_{1}}\subseteq\mathcal{O}_{v_{2}}),$$
for all $(K_{1},v_{1}),(K_{2},v_{2})\in\mathcal{K}$.
\item The valuation ring is uniformly $\forall$-$C$-definable in $\mathcal{K}$
if and only if
$$(K_{1}\subseteq K_{2}\implies\mathcal{O}_{v_{2}}\cap K_{1}\subseteq\mathcal{O}_{v_{1}}),$$
for all $(K_{1},v_{1}),(K_{2},v_{2})\in\mathcal{K}$.
\end{enumerate}
\end{Proposition}

\begin{Remark}
In subsequent arguments, rather than use the `only if' direction of \ref{thm:Prestel} we will instead use the basic fact that existential sentences `pass up'.
\end{Remark}

We conclude this section with a probably well-known description of sets defined by quantifier-free formulas:

\begin{Lemma}\label{Lemma:qfdefinable}
Let $(K,v)$ be a valued $C$-field and $\phi(x)$ a quantifier-free $\mathcal{L}_{\rm vf}(C)$-formula
with free variable $x$ belonging to the field sort.
Then the set defined by $\phi(x)$ in $K$ is of the form
$\phi(K)=A\cup U$ 
with $A$ finite and algebraic over $C_K$, and $U$ open.
\end{Lemma}

\begin{proof}
Let $(K^{\mathrm{alg}},v)$ be an algebraic closure of $(K,v)$, viewed as a valued $C$-field.
Since $\phi(x)$ is quantifier-free, $\phi(K)=\phi(K^{\mathrm{alg}})\cap K$.
Let $\mathcal{L}_{\rm div}=\mathcal{L}_{\rm ring}\cup\{|\}$ be the expansion of $\mathcal{L}_{\mathrm{ring}}$ by the divisibility predicate
$x|y\Leftrightarrow v(x)\leq v(y)$.
There exists an $\mathcal{L}_{\rm div}(C)$-formula $\psi(x)$ such that $\psi(K^{\mathrm{alg}})=\phi(K^{\mathrm{alg}})$.
By \cite[Theorem 3.26]{Holly95}, the set $\psi(K^{\mathrm{alg}})$ defined by $\psi(x)$ is a boolean combination of balls and singletons
(in fact, it is a finite union of ``Swiss cheeses'' in Holly's picturesque terminology).
Moreover, the singletons are algebraic over $C_{K^{\mathrm{alg}}}=C_{K}$;
this can be seen either by a straightforward analysis of $\mathcal{L}_{\mathrm{div}}(C)$-formulas in $(K^{\mathrm{alg}},v)$,
or by the fact that `model theoretic algebraic closure' in $(K^{\mathrm{alg}},v)$ is equal to `field theoretic algebraic closure'. 
Therefore $\phi(K^{\mathrm{alg}})=\psi(K^{\mathrm{alg}})$ is equal to $A'\cup U'$ for some open set $U'$ and a finite set $A'$ which is algebraic over $C_{K}$.
Since quantifier-free formulas `pass down', $\phi(K)$ is of the required form.
\end{proof}

\section{Characterizing uniform definitions}
\label{sec:embedded}

\noindent
We fix a ring $C$ and a class of $C$-fields $\mathcal{F}$.

\begin{Definition}
Let
$$
 \mathcal{H}(\mathcal{F})=\left\{(K,v) \;|\; (K,v)\mbox{ henselian valued $C$-field}, Kv\in\mathcal{F} \right\}
$$ 
denote the class of henselian valued $C$-fields with residue field in $\mathcal{F}$,
and
$$
 \mathcal{H}'(\mathcal{F})=\left\{(K,v)\in\mathcal{H}(\mathcal{F}) \;|\; v\mbox{ nontrivial}\right\}
$$ 
the subclass of nontrivially valued fields.
Moreover, let
$$
 \mathcal{H}_e(\mathcal{F})=\{(K,v)\in\mathcal{H}(\mathcal{F})\;|\;\mathrm{char}(K)=\mathrm{char}(Kv)\}
$$ 
be the subclass of equicharacteristic fields,
and
$$
 \mathcal{H}_0(\mathcal{F})=\{(K,v)\in\mathcal{H}(\mathcal{F})\;|\;\mathrm{char}(K)=0\}
$$ 
the subclass of fields of characteristic zero,
and define $\mathcal{H}_e'(\mathcal{F})$ and $\mathcal{H}_0'(\mathcal{F})$ accordingly.
For a single $C$-field $F$, we let
$\mathcal{H}(F)$ be $\mathcal{H}(\mathcal{F})$ with $\mathcal{F}$ the class of $C$-fields elementarily equivalent to $F$.
Again we define $\mathcal{H}'(F)$, $\mathcal{H}_{e}(F)$, $\mathcal{H}_{0}(F)$, $\mathcal{H}_{e}'(F)$, and $\mathcal{H}_{0}'(F)$ accordingly.
\end{Definition}

\begin{Remark}
We have that
$\mathcal{H}(\mathcal{F})=\mathcal{H}_e(\mathcal{F})\cup\mathcal{H}_0(\mathcal{F})$
and
$\mathcal{H}'(\mathcal{F})=\mathcal{H}_e'(\mathcal{F})\cup\mathcal{H}_0'(\mathcal{F})$.
Note that if $\mathcal{F}$ is an $\mathcal{L}_{\rm ring}(C)$-elementary class of $C$-fields,
then the classes $\mathcal{H}(\mathcal{F})$, $\mathcal{H}_e(\mathcal{F})$, $\mathcal{H}_0(\mathcal{F})$,
$\mathcal{H}'(\mathcal{F})$, $\mathcal{H}_e'(\mathcal{F})$ and $\mathcal{H}_0'(\mathcal{F})$
are $\mathcal{L}_{\rm vf}(C)$-elementary classes of valued $C$-fields.
\end{Remark}

\begin{Remark}
In \cite{Prestel}, \ref{thm:Prestel} is used to reprove \cite[Theorem 1.1]{Fehm}: there is a uniform $\exists$-$\emptyset$-definition of the valuation ring for henselian fields with residue field elementarily equivalent to a fixed finite or PAC field (not containing an algebraically closed subfield).
We extend this idea to characterize those elementary classes of $C$-fields $\mathcal{F}$ for which the valuation ring and the valuation ideal are uniformly $\exists$-$C$-definable in $\mathcal{H}(\mathcal{F})$, in $\mathcal{H}_e(\mathcal{F})$, and (under certain conditions on $C$) in $\mathcal{H}_0(\mathcal{F})$.
\end{Remark}

\begin{Remark}
Most definitions and statements in this section can be phrased both in an ``existential'' and in a ``universal'' setting
(see also \ref{rem:EAOm}), 
and both for ``arbitrary'' valuations and for ``equicharacteristic'' valuations. 
The ``existential'' and ``universal'' definitions and statements show a huge formal similarity, but the proofs are sometimes different for both cases.
On the other hand, the proofs for the  ``arbitrary'' and ``equicharacteristic'' versions are so similar that we combine them by putting 
the necessary changes in the equicharacteristic case in brackets. 
\end{Remark}

\subsection{Embedded residue and large classes of fields}

\begin{Definition}\label{def:embedded.residue.large}
\begin{enumerate}
\item
We say that $\mathcal{F}$ has \em [equicharacteristic] embedded residue \rm 
if there exist $F_{1},F_{2}\in\mathcal{F}$ and a nontrivial [equicharacteristic] valuation $w$ on $F_{1}$ with an embedding of $F_1w$ into $F_2$.
For a single $C$-field $F$, we say $F$ has \em embedded residue \rm if the class of $C$-fields elementarily equivalent to $F$ has embedded residue.
\item
We say that $\mathcal{F}$ is \em [equicharacteristic] large \rm if there exist $F_{1},F_{2}\in\mathcal{F}$, a $C$-subfield $E\subseteq F_{2}$, and a nontrivial [equicharacteristic] henselian valuation $w$ on $E$ such that
$Ew$ is isomorphic to $F_{1}$.
For a single $C$-field $F$, we say $F$ is 
{\em large\footnote{This notion of a large $C$-field is related to the notion of a large field in the sense of Pop \cite{Pop96}. We discuss this connection in Section \ref{sec:large}.}} 
if the class of $C$-fields elementarily equivalent to $F$ is large.
\end{enumerate}
\end{Definition}

\begin{Remark}\label{rem:unmixed}
Note that if $\mathcal{F}$ has equicharacteristic embedded residue,
then $\mathcal{F}$ has embedded residue.
Moreover, note that if $\mathcal{F}$ is {\em unmixed}, i.e.~
$\mathcal{F}$ does not contain both fields of characteristic zero and of positive characteristic,
then $\mathcal{F}$ has embedded residue if and only if it has equicharacteristic embedded residue.
This applies in particular if $\mathcal{F}$ is a class of $C$-fields elementarily equivalent to a fixed $C$-field $F$;
in other words, if $\mathcal{F}$ is an elementary class of $C$-fields and one $F\in\mathcal{F}$ has embedded residue,
then $\mathcal{F}$ has equicharacteristic embedded residue.
In the case that $C$ itself has positive characteristic or is a field, every class of $C$-fields is unmixed.
The analogous statements hold for ``large'' instead of ``embedded residue''.
\end{Remark}

\begin{Lemma}\label{lemma:Schoutens}
Let $F$ be a $C$-field.
\begin{enumerate}
\item $F$ has embedded residue iff there is an elementary extension $F\preceq F^*$
and a nontrivial valuation $v$ on $F^*$ 
with an embedding of $F^*v$ into $F^*$.
\item $F$ is large iff there is an elementary extension $F\preceq F^*$,
a subfield $E\subseteq F^*$ and a nontrivial henselian valuation $v$ on $E$
such that $Ev$ is isomorphic to $F^*$.
\end{enumerate}
\end{Lemma}

\begin{proof}
(1)
Let $F_1\equiv F_2\equiv F$ and $v$ a nontrivial valuation on $F_1$ such that $F_1v$ embeds into $F_2$.
By \cite{Shelah} there exists a cardinal $\lambda$ and an ultrafilter $D$ on $\lambda$ such that
the ultrapowers $F_1^\lambda/D$ and $F_2^\lambda/D$ are isomorphic and $|F|^+$-saturated.
Thus, if we let $(F^*,v^*)=(F_1,v)^\lambda/D$, then $F$ can be embedded elementarily into $F^*$,
and $F^*v^*\cong (F_1v)^\lambda/D$ embeds into $F_2^\lambda/D\cong F^*$.

(2)
Let $F_1\equiv F_2\equiv F$, $E$ a subfield of $F_1$ and $v$ a nontrivial henselian valuation on $E$ with $Ev\cong F_2$.
Again there exist $\lambda$ and $D$ such that $F_1^\lambda/D\cong F_2^\lambda/D$ is $|F|^+$-saturated.
If we equip $F_1$ with a predicate for $E$ and let $(F^*,E^*,v^*)=(F_1,E,v)^\lambda/D$, 
then $E^*$ is a subfield of $F^*$ and $v^*$ is a nontrivial henselian valuation on $E^*$
such that $E^*v^*=(Ev)^\lambda/D\cong F_2^\lambda/D\cong F^*$.
\end{proof}

\begin{Lemma}\label{lem:sep.cl}
Let $F$ be a $C$-field
that contains an algebraically closed $C$-field $D$.
Then
\begin{enumerate}
\item $F$ has embedded residue.
\item $F$ is large.
\end{enumerate}
\end{Lemma}

\begin{proof}
(1) By passing to an elementary extension, we may assume that $F$ is transcendental over $D$. 
By \ref{lem:extend.with.alg.res.field} there exists an extension of 
the trivial valuation on $D$ to a valuation $v$ on $F$ with $Fv= D\subset F$.
It is necessarily nontrivial.

(2) Let $(E,v)$ be a nontrivial henselian extension of the trivially valued field $F$ with $Ev\cong F$ (\ref{rem:henselian.ext}).
Since $D$ is algebraically closed, $D\preceq_\exists E$. 
Let $(F^*,D^*)$ be an $|E|^+$-saturated elementary extension of $F$ with a predicate for $D$.
Then there exists a $D$-embedding $E\rightarrow D^*\subseteq F^*$.
\end{proof}

\subsection{Existential definitions}

\begin{Lemma}\label{lem:ER.implies.not.uniform}
Suppose that $\mathcal{F}$ has [equicharacteristic] embedded residue. 
Then the valuation ring \rm\bf is not \rm\em uniformly $\exists$-$C$-definable in $\mathcal{H}'(\mathcal{F})$ [respectively, $\mathcal{H}_e'(\mathcal{F})$].
\end{Lemma}

\begin{proof}
By assumption (in both cases), there exist $F_{1},F_{2}\in\mathcal{F}$ and a nontrivial valuation $w$ on $F_{1}$ such that 
$F_{1}w\subseteq F_{2}$. 
Let $(K_1,v_1)$ be a nontrivial henselian extension of the trivially valued field $F_1$ with $K_1v_1= F_{1}$ (\ref{rem:henselian.ext}).
Then $(K_1,v_1)\in\mathcal{H}_e'(\mathcal{F})$.
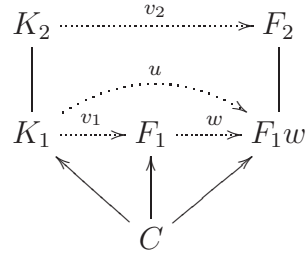
$$
\xymatrix{
 K_2\ar@{-}[d]\ar@{.>}[rr]^{v_2} && F_2\ar@{-}[d] \\
 K_1\ar@{.>}[r]^{v_1}\ar@/^1.6pc/@{.>}[rr]^{u} & F_1\ar@{.>}[r]^w & F_1w\\
 &C\ar@{->}[ul]\ar@{->}[u]\ar@{->}[ur]&
}
$$
Let $u$ denote the composition $w\circ v_1$ on $K_1$. 
Observe that $u$ is a proper refinement of $v_1$ and that $K_1u=F_{1}w$ is a subfield of $F_{2}$. 
By \ref{lem:complete.the.square} there exists a henselian extension $(K_2,v_2)$ of $(K_1,u)$ such that $K_2v_2$ is isomorphic to $F_{2}$. 

Then $(K_2,v_2)\in\mathcal{H}'(\mathcal{F})$,
but $v_2$ restricted to $K_1$ is a proper refinement of $v_1$,
hence $\mathcal{O}_{v_1}\not\subseteq\mathcal{O}_{v_2}$.
Since existential formulas `pass up', the valuation ring is not uniformly $\exists$-$C$-definable in $\mathcal{H}'(\mathcal{F})$.

If in addition $w$ is equicharacteristic, then 
so is $u$, and therefore also $v_2$,
hence
$(K_2,v_2)\in\mathcal{H}_e'(\mathcal{F})$, showing that
the valuation ring is not uniformly $\exists$-$C$-definable in $\mathcal{H}_e'(\mathcal{F})$.
\end{proof}

\begin{Lemma}\label{lem:no.ER.implies.uniform}
Suppose that $\mathcal{F}$ is an elementary class and that $\mathcal{F}$ does not have [equicharacteristic] embedded residue. 
Then the valuation ring \rm\bf is \rm\em uniformly $\exists$-$C$-definable in $\mathcal{H}(\mathcal{F})$ [respectively, $\mathcal{H}_e(\mathcal{F})$].
\end{Lemma}

\begin{proof}
We test the hypothesis of \ref{thm:Prestel} for the class
$\mathcal{K}=\mathcal{H}(\mathcal{F})$ [respectively, $\mathcal{K}=\mathcal{H}_e(\mathcal{F})$]. 

Let $(K_1,v_1),(K_2,v_2)\in\mathcal{H}(\mathcal{F})$
and suppose that $K_1$ is an $\mathcal{L}_{\mathrm{ring}}(C)$-substructure of $K_2$. 
Let $L\subseteq K_2$ be the henselisation of $K_1$ with respect to the restriction $u$ of $v_2$ to $K_1$. 
Let $w_1$ be the unique extension of $v_1$ to $L$ and $w_2$ the restriction of $v_2$ to $L$.
We now have a trichotomy by comparing $v_1$ and $u$: either $v_1$ and $u$ are incomparable,
or $v_{1}$ is a proper coarsening of $u$, or $v_{1}$ is a refinement of $u$.
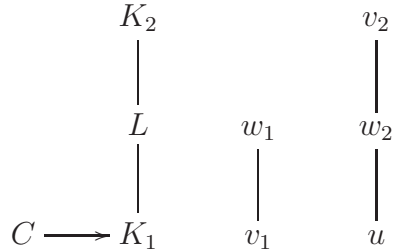
$$
\xymatrix{
 &K_2\ar@{-}[d] & & v_2\ar@{-}[d] \\
 &L\ar@{-}[d] & w_1\ar@{-}[d] & w_2\ar@{-}[d]\\
 C\ar@{->}[r]&K_1 & v_1 & u 
}
$$
If $v_1$ and $u$ are incomparable,
then so are $w_1$ and $w_2$, hence the residue fields of $w_1$, $w_2$, and their finest common coarsening $w$ are all separably closed. 
Furthermore $w_2$ induces a nontrivial valuation $\bar{w}_2$ on the separably closed field $Lw$. 
By \ref{lem:residue.of.sep.closed}, $(Lw)\bar{w}_2=Lw_2$ is algebraically closed,
hence, by \ref{lem:sep.cl}, $K_2v_2\supseteq Lw_2$ has embedded residue.
Thus, $\mathcal{F}$ has equicharacteristic embedded residue, see \ref{rem:unmixed}. 
This contradicts our assumption (in both cases). 
Therefore $v_1$ and $u$ are comparable.

If $v_{1}$ is a proper coarsening of $u$ then $u$ induces a nontrivial valuation $\bar{u}$ on $K_1v_1$ and 
$(K_1v_1)\bar{u}=K_1u\subseteq K_2v_2$. 
Thus $\mathcal{F}$ has embedded residue.
If in addition, 
$(K_2,v_2)\in\mathcal{H}_e(\mathcal{F})$, 
then $u$ is equicharacteristic,
and then so is $\bar{u}$,
hence $\mathcal{F}$ has equicharacteristic embedded residue.
Thus, in both cases this is a contradiction.

Therefore $v_{1}$ is a refinement of $u$, i.e. $\mathcal{O}_{v_{1}}\subseteq\mathcal{O}_{u}$, as required.
Applying \ref{thm:Prestel}, we are done.
\end{proof}

\begin{Theorem}[Characterisation Theorem, $\exists$-case]\label{thm:existential.characterisation}
Let $\mathcal{F}$ be an elementary class of $C$-fields. The following are equivalent:
\begin{enumerate}
\item[$(1^{'\exists})$] The valuation ring is uniformly $\exists$-$C$-definable in $\mathcal{H}'(\mathcal{F})$.
\item[$(1^{\exists})$] The valuation ring is uniformly $\exists$-$C$-definable in $\mathcal{H}(\mathcal{F})$.
\item[$(2^{\exists})$] $\mathcal{F}$ does not have embedded residue.
\end{enumerate}
Moreover, also the following are equivalent:
\begin{enumerate}
\item[$(1^{'\exists}_e)$] The valuation ring is uniformly $\exists$-$C$-definable in $\mathcal{H}'_e(\mathcal{F})$.
\item[$(1^{\exists}_e)$] The valuation ring is uniformly $\exists$-$C$-definable in $\mathcal{H}_e(\mathcal{F})$.
\item[$(2^{\exists}_e)$] $\mathcal{F}$ does not have equicharacteristic embedded residue.
\end{enumerate}
\end{Theorem}

\begin{proof}
$(1^{\exists})\implies(1^{'\exists})$ and $(1^{\exists}_{e})\implies(1^{'\exists}_{e})$:
This follows from the inclusions $\mathcal{H}^{'}(\mathcal{F})\subseteq\mathcal{H}(\mathcal{F})$
and $\mathcal{H}_e^{'}(\mathcal{F})\subseteq\mathcal{H}_e(\mathcal{F})$.

$(1^{'\exists})\implies(2^{\exists})$ and $(1^{'\exists}_{e})\implies(2^{\exists}_e)$:
Apply \ref{lem:ER.implies.not.uniform}.

$(2^{\exists})\implies(1^{\exists})$ and $(2^{\exists}_{e})\implies(1^{\exists}_{e})$:
Apply \ref{lem:no.ER.implies.uniform}.
\end{proof}

\begin{Corollary}\label{cor:existential.characterisation.one.characteristic}
Let $\mathcal{F}$ be an elementary class of $C$-fields which is unmixed.
Then all six conditions $(1^{'\exists})$, $(1^{\exists})$, $(2^\exists)$, $(1^{'\exists}_{e})$, $(1^{\exists}_{e})$, and $(2^{\exists}_{e})$ are equivalent.
\end{Corollary}

\begin{Remark}
The equivalence between $(1^{'\exists})$ and $(1^{\exists})$
(resp.~$(1^{'\exists}_{e})$ and $(1^{\exists}_{e})$)
does not seem obvious.
In general we may require a different definition:
an $\exists$-$\mathcal{L}_{\mathrm{ring}}(C)$-formula that uniformly defines the valuation ring in $\mathcal{H}^{'}(\mathcal{F})$ may not also define the trivial valuation ring in the trivially valued $C$-fields with residue field in $\mathcal{F}$.
For example, let $C=\mathbb{Z}$ and $\mathcal{F}=\{\mathbb{F}_{p}\}$, for a prime $p$.
As discussed below in \ref{prop:positive.applications}(1), $\mathcal{F}$ does not have embedded residue.
Therefore there is an $\exists$-$\mathcal{L}_{\mathrm{ring}}(C)$-formula $\phi(x)$ that uniformly defines the valuation ring in $\mathcal{H}(\mathcal{F})$.
However, it is easy to see that the formula
$$
 \phi(x)\wedge\exists y_{0},...,y_{p}\;\bigwedge_{i\neq j}y_{i}\neq y_{j}
$$
also uniformly defines the valuation ring in $\mathcal{H}^{'}(\mathcal{F})$ but does not define the trivial valuation ring in $\mathbb{F}_{p}$.

Also note that as explained in \ref{rem:unmixed}, the assumption of the corollary is satisfied in particular 
if $\mathcal{F}$ consists of $C$-fields elementarily equivalent to a fixed $C$-field $F$,
and also if $C$ has positive characteristic or is a field.
If we put a different condition on $C$, then we get a further equivalent characterization:
\end{Remark}

\begin{Corollary}\label{cor:existential.characterisation.characteristic.zero}
Let $C$ be a Dedekind domain\footnote{It would suffice that for every $\mathfrak{p}\in{\rm Spec}(C)$ the local ring $(C_\mathfrak{p},\mathfrak{p}C_\mathfrak{p})$
is dominated by a valuation ring $(\mathcal{O},\mathfrak{m})$ with $\mathcal{O}/\mathfrak{m}=C_\mathfrak{p}/\mathfrak{p}C_\mathfrak{p}$.
This condition is satisfied more generally for Pr\"ufer domains (like the ring of algebraic integers $C=\mathbb{Z}^{\rm alg}$) and for regular Noetherian domains (like $C=\mathbb{Z}[t]$).} 
of characteristic zero
and $\mathcal{F}$ an elementary class of $C$-fields.
Then $(1^{'\exists})$, $(1^{\exists})$, and $(2^{\exists})$ are each equivalent to the following:
\begin{enumerate}
\item[$(1^{\exists}_0)$] The valuation ring is uniformly $\exists$-$C$-definable in $\mathcal{H}_0(\mathcal{F})$.
\end{enumerate}
\end{Corollary}

\begin{proof}
$(1^{'\exists})$ $(1^{\exists})$, and $(2^{\exists})$ are equivalent by \ref{thm:existential.characterisation}, and trivially $(1^{\exists})\implies (1^{\exists}_0)$.
We repeat and adapt the proof of $(1^{\exists})\implies(2^{\exists})$ (i.e. the proof of \ref{lem:ER.implies.not.uniform}) to obtain a proof of $(1^{\exists}_0)\implies(2^{\exists})$: 

Suppose there exist $F_{1},F_{2}\in\mathcal{F}$ and a nontrivial valuation $w$ on $F_{1}$ such that 
$F_{1}w\subseteq F_{2}$. 
Instead of taking an extension $K_1$ of $F_1$,
we note that the localization of $C$ at the kernel $\mathfrak{p}$ of the structure homomorphism $C\rightarrow F_1$
is the valuation ring of a valuation $v_0$ on a $C$-field $K_0$ of characteristic zero
with residue field $K_0v_0\cong {\rm Quot}(C/\mathfrak{p})\subseteq F_1$.
By \ref{lem:complete.the.square}, there is a henselian extension $(K_1,v_1)$ of $(K_0,v_0)$ with $K_1v_1\cong F_1$,
so $(K_1,v_1)\in\mathcal{H}_0(\mathcal{F})$.

Let $u$ denote the composition $w\circ v_1$ on $K_1$. 
Observe that $u$ is a proper refinement of $v_1$ and that $K_1u=F_{1}w$ is a subfield of $F_{2}$. 
By \ref{lem:complete.the.square} there exists a henselian extension $(K_2,v_2)$ of $(K_1,u)$ such that $K_2v_2$ is isomorphic to $F_{2}$.
Then $(K_2,v_2)\in\mathcal{H}_0(\mathcal{F})$,
but $v_2$ restricted to $K_1$ is a proper refinement of $v_1$,
hence $\mathcal{O}_{v_1}\not\subseteq\mathcal{O}_{v_2}$.
Since existential formulas `pass up', the valuation ring is not uniformly $\exists$-$C$-definable in $\mathcal{H}_0(\mathcal{F})$.
\end{proof}

\begin{Remark}
Note that \ref{cor:existential.characterisation.characteristic.zero} applies for example to $C=\mathbb{Z}$,
and $\exists$-$\mathbb{Z}$-definable is the same as $\exists$-$\emptyset$-definable.
Moreover, if we apply \ref{cor:existential.characterisation.one.characteristic} to $C=\mathbb{Z}/n\mathbb{Z}$ for $n\in\mathbb{N}$,
then $\mathcal{F}$ can be any elementary class of fields $F$ with ${\rm char}(F)$ dividing $n$,
and here again, $\exists$-$\mathbb{Z}/n\mathbb{Z}$-definable is the same as $\exists$-$\emptyset$-definable.
\end{Remark}

\subsection{Universal definitions}

\begin{Lemma}\label{lem:large.implies.not.uniform}
Suppose that $\mathcal{F}$ is [equicharacteristic] large. 
Then the valuation ring \rm\bf is not \rm\em uniformly $\forall$-$C$-definable in $\mathcal{H}'(\mathcal{F})$
[respectively, $\mathcal{H}_e'(\mathcal{F})$].
\end{Lemma}

\begin{proof}
By assumption (in both cases), there exist $F_{1},F_{2}\in\mathcal{F}$, a subfield $E\subseteq F_{2}$ and a nontrivial henselian valuation $w$ on $E$ such that $Ew$ is isomorphic to $F_{1}$. 
Let $(K_1,u)$ be a nontrivial henselian extension of the trivially valued field $E$ with $K_1u=E$ (\ref{rem:henselian.ext}).
Let $v_1$ denote the composition $w\circ u$, so that $K_1v_1=F_{1}$. 
Then $(K_1,v_1)\in\mathcal{H}'(\mathcal{F})$. 
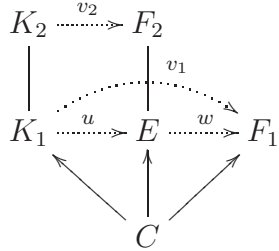
$$
\xymatrix{
 K_2\ar@{-}[d]\ar@{.>}[r]^{v_2} & F_2\ar@{-}[d]&\\
 K_1\ar@{.>}[r]^{u}\ar@/^1.6pc/@{.>}[rr]^{\qquad v_{1}} & E\ar@{.>}[r]^w & F_1\\
 &C\ar@{->}[ul]\ar@{->}[u]\ar@{->}[ur]
}
$$
Using \ref{lem:complete.the.square}, we may construct a henselian extension $(K_2,v_2)$ of $(K_1,u)$ of valued fields so that $K_2v_2$ is isomorphic to $F_{2}$. 
Thus $(K_2,v_2)\in\mathcal{H}_e'(\mathcal{F})$ and $K_1\subseteq K_2$. 
Since $w$ is nontrivial, the restriction of $v_2$ to $K_1$ is a proper coarsening of $v_1$.
Since universal formulas `go down', the valuation ring is not uniformly $\forall$-$C$-definable in $\mathcal{H}'(\mathcal{F})$.

If in addition $w$ is equicharacteristic, then so is $v_1$, hence $(K_1,v_1)\in\mathcal{H}_e'(\mathcal{F})$,
showing that
the valuation ring is not uniformly $\forall$-$C$-definable in $\mathcal{H}_e'(\mathcal{F})$.
\end{proof}

\begin{Lemma}\label{lem:not.large.implies.uniform}
Suppose that $\mathcal{F}$ is an elementary class and that $\mathcal{F}$ is not [equicharacteristic] large. Then the valuation ring \rm\bf is \rm\em uniformly $\forall$-$C$-definable in $\mathcal{H}(\mathcal{F})$
[respectively, in $\mathcal{H}_e(\mathcal{F})$].
\end{Lemma}

\begin{proof}
We test the hypothesis of \ref{thm:Prestel}. 
Let $(K_1,v_1),(K_2,v_2)\in\mathcal{H}(\mathcal{F})$ and suppose that $K_1$ is an $\mathcal{L}_{\mathrm{ring}}(C)$-substructure of $K_2$. 
Let $u$ denote the restriction of $v_2$ to $K_1$.
As in the proof of \ref{lem:no.ER.implies.uniform} we see that if $v_1$ and $u$ are incomparable, then
$K_2v_2$ would contain an algebraically closed $C$-field.
Since by \ref{lem:sep.cl} this would also imply that $K_2v_2$ is large,
and therefore $\mathcal{F}$ is equicharacteristic large  (\ref{rem:unmixed}), which is a contradiction (in both cases),
we conclude that $v_1$ and $u$ are comparable.

If $v_{1}$ is a proper refinement of $u$ then $v_1$ induces a nontrivial henselian valuation $\bar{v}_1$ on $K_1u\subseteq K_2v_2$ and $(K_1u)\bar{v}_1=K_1v_1$. 
Thus $\mathcal{F}$ is large.
If in addition $(K_1,v_1)\in\mathcal{H}_e(\mathcal{F})$,
then also $\bar{v}_1$ is equicharacteristic, so $\mathcal{F}$ is equicharacteristic large.
Thus, in both cases this is a contradiction.
Therefore $v_{1}$ is a coarsening of $u$, i.e. $\mathcal{O}_{v_{1}}\supseteq\mathcal{O}_{u}$, as required.
Applying \ref{thm:Prestel}, we are done.
\end{proof}

\begin{Theorem}[Characterisation Theorem, $\forall$-case]\label{thm:universal.characterisation}
Let $\mathcal{F}$ be an elementary class of $C$-fields. The following are equivalent:
\begin{enumerate}
\item[$(1^{'\forall})$] The valuation ring is uniformly $\forall$-$C$-definable in $\mathcal{H}'(\mathcal{F})$.
\item[$(1^{\forall})$] The valuation ring is uniformly $\forall$-$C$-definable in $\mathcal{H}(\mathcal{F})$.
\item[$(2^{\forall})$] $\mathcal{F}$ is not large.
\end{enumerate}
Moreover, also the following are equivalent:
\begin{enumerate}
\item[$(1^{'\forall}_e)$] The valuation ring is uniformly $\forall$-$C$-definable in $\mathcal{H}'_e(\mathcal{F})$.
\item[$(1^{\forall}_e)$] The valuation ring is uniformly $\forall$-$C$-definable in $\mathcal{H}_e(\mathcal{F})$.
\item[$(2^{\forall}_e)$] $\mathcal{F}$ is not equicharacteristic large.
\end{enumerate}
\end{Theorem}

\begin{proof}
$(1^{\forall})\implies(1^{'\forall})$ and $(1^{\forall}_{e})\implies(1^{'\forall}_{e})$:
This follows from the inclusions $\mathcal{H}'(\mathcal{F})\subseteq\mathcal{H}(\mathcal{F})$
and $\mathcal{H}_e'(\mathcal{F})\subseteq\mathcal{H}_e(\mathcal{F})$.

$(1^{'\forall})\implies(2^{\forall})$ and $(1^{'\forall}_{e})\implies(2^{\forall}_{e})$:
Apply \ref{lem:large.implies.not.uniform}.

$(2^{\forall})\implies(1^{\forall})$ and $(2^{\forall}_e)\implies(1^{\forall}_e)$:
Apply \ref{lem:not.large.implies.uniform}.
\end{proof}

\begin{Corollary}\label{cor:universal.characterisation.one.characteristic}
Let $\mathcal{F}$ be an elementary class of $C$-fields which is unmixed.
Then all six conditions $(1^{'\forall})$, $(1^{\forall})$, $(2^{\forall})$, $(1^{'\forall}_{e})$, $(1^{\forall}_{e})$, and $(2^{\forall}_{e})$ are equivalent.
\end{Corollary}

\begin{Corollary}\label{cor:universal.characterisation.characteristic.zero}
Let $C$ be a Dedekind domain of characteristic zero
and $\mathcal{F}$ an elementary class of $C$-fields.
Then $(1^{'\forall})$, $(1^{\forall})$, and $(2^{\forall})$ are each equivalent to the following:
\begin{enumerate}
\item[$(1^{\forall}_0)$] The valuation ring is uniformly $\forall$-$C$-definable in $\mathcal{H}_0(\mathcal{F})$.
\end{enumerate}
\end{Corollary}

\begin{proof}
$(1^{'\forall})$, $(1^{\forall})$, and $(2^{\forall})$ are equivalent by \ref{thm:universal.characterisation}, and trivially $(1^{\forall})\implies(1^{\forall}_0)$.
We repeat and adapt the proof of $(1^{\forall})\implies(2^{\forall})$
(i.e. the proof of \ref{lem:large.implies.not.uniform}) to obtain a proof of $(1^{\forall}_0)\implies(2^{\forall})$: 

Suppose there exist $F_{1},F_{2}\in\mathcal{F}$, a subfield $E\subseteq F_{2}$ and a nontrivial henselian valuation $w$ on $E$ such that $Ew$ is isomorphic to $F_{1}$. 
The localization of $C$ at the kernel $\mathfrak{p}$ of the structure homomorphism $C\rightarrow E$
is the valuation ring of a valuation $v_0$ on a $C$-field $K_0$ of characteristic zero
with residue field $K_0v_0\cong {\rm Quot}(C/\mathfrak{p})\subseteq E$.
By \ref{lem:complete.the.square}, there is a henselian extension $(K_1,u)$ of $(K_0,v_0)$ with $K_1u\cong E$.
Let $v_1$ denote the composition $w\circ u$, so that $K_1v_1=F_{1}$. Then $(K_1,v_1)\in\mathcal{H}_0(\mathcal{F})$. 

Using \ref{lem:complete.the.square}, we may construct a henselian extension $(K_2,v_2)$ of $(K_1,u)$ of valued fields so that $K_2v_2$ is isomorphic to $F_{2}$. 
Thus $(K_2,v_2)\in\mathcal{H}_0(\mathcal{F})$ and $K_1\subseteq K_2$. 
Since $w$ is nontrivial, the restriction of $v_2$ to $K_1$ is a proper coarsening of $v_1$.
Since universal formulas `go down', the valuation ring is not uniformly $\forall$-$C$-definable in $\mathcal{H}_0(\mathcal{F})$.
\end{proof}


\section{Making definitions uniform for equicharacteristic fields}
\label{sec:uniform}

\noindent
The goal of this section is to show that existential definitions of the valuation ring or valuation ideal
of an equicharacteristic henselian nontrivially valued field can be modified to work
for all such valued fields with elementarily equivalent residue field.
Depending on the parameters $C$, we may have to restrict the residue fields that we consider:

\begin{Definition}
We say that a $C$-field $F$ satisfies $(*)$ if
\renewcommand{\theequation}{$*$}
\begin{equation}\label{star}
\begin{minipage}{11cm}
{\it \begin{enumerate} 
      \item[(a)] $C$ is integral over its prime ring, \underline{or} 
      \item[(b)] $C$ is a perfect field and $F$ is perfect.
     \end{enumerate}
}
\end{minipage}
\end{equation}
We say that a class of $C$-fields $\mathcal{F}$ satisfies $(*)$ if each $F\in\mathcal{F}$ satisfies $(*)$.
\end{Definition}

Note that in the case $C=\mathbb{Z}$ (that is, when we consider $\exists$-$\emptyset$-definitions),
all classes of $C$-fields satisfy $(*)$.

\begin{Proposition}\label{prp:GDM.version.1}
Let $\mathcal{F}$ be an elementary class of $C$-fields that satisfies $(*)$.
Let $\phi(x)$ be an $\exists$-$\mathcal{L}_{\mathrm{ring}}(C)$-formula.
\begin{enumerate}
\item  If for all $F\in\mathcal{F}$ there exists $(K,v)\in\mathcal{H}_e'(F)$ with $\phi(K)=\mathcal{O}_v$,
then there is an $\exists$-$\mathcal{L}_{\mathrm{ring}}(C)$-formula $\psi(x)$ with 
$\psi(K)=\mathcal{O}_v$
for all $(K,v)\in\mathcal{H}_e'(\mathcal{F})$.
\item  If for all $F\in\mathcal{F}$ there exists $(K,v)\in\mathcal{H}_e'(F)$ with $\phi(K)=\mathfrak{m}_v$,
then there is an $\exists$-$\mathcal{L}_{\mathrm{ring}}(C)$-formula $\chi(x)$ with 
$\chi(K)=\mathfrak{m}_v$
for all $(K,v)\in\mathcal{H}_e'(\mathcal{F})$.
\end{enumerate}
\end{Proposition}

\begin{Remark}
We note that in general we cannot take $\psi(x)$ (resp.~$\chi(x)$) to be simply $\phi(x)$, as the following example 
in the $\mathfrak{m}$-case shows:
If the $\exists$-$\mathcal{L}_{\rm ring}(C)$-formula $\chi(x)$ 
uniformly defines the valuation ideal in $\mathcal{H}_e'(F)$, then the $\exists$-$\mathcal{L}_{\rm ring}(C)$-formula $\phi(x)$ given by
$$
 \exists y \exists z (x=yz\wedge\chi(y)\wedge\chi(z))
$$
defines the valuation ideal in any $(K,v)\in\mathcal{H}_e'(F)$ with divisible value group,
but for example not in $(F((t)),v_t)\in\mathcal{H}_e'(F)$.
The proof of \ref{prp:GDM.version.1} below however does give a quite explicit construction of $\psi$ and $\chi$ from $\phi$.
\end{Remark}

\begin{Remark}
Note that the conclusion of (1) (respectively, (2)) is $(1'^{\exists}_e)$ (resp., $(1'^{\forall}_e)$) from \ref{thm:existential.characterisation} (resp.\ \ref{thm:universal.characterisation}), see also \ref{rem:EAOm}.
In order to prove this proposition, we first break the statement $\phi(K)=\mathcal{O}_v$ (respectively $\phi(K)=\mathfrak{m}_v$)
into several parts which we then treat separately:
\end{Remark}

\begin{Definition}\label{def:important.sentences}
For an $\mathcal{L}_{\rm vf}(C)$-formula $\phi(x)$ 
with free variable $x$ belonging to the field sort
we define the following
$\mathcal{L}_{\rm vf}(C)$-sentences:
\begin{enumerate}
\item[$({\rm SO}_\phi)$]  $\forall x\;(\phi(x)\rightarrow v(x)\geq0)$
\item[$({\rm CO}_\phi)$]  $\forall x\;(v(x)\geq0 \rightarrow \phi(x))$
\item[$({\rm SM}_\phi)$]  $\forall x\;(\phi(x)\rightarrow v(x)>0)$
\item[$({\rm CM}_\phi)$] $\forall x\;(v(x)>0\rightarrow\phi(x))$
\item[$({\rm IM}_\phi)$]  $\exists x\;(v(x)>0\wedge x\neq0\wedge\phi(x))$
\item[$({\rm R}_\phi)$] $\forall^{k}x\exists y\;(\mathrm{res}(y)=x\wedge\phi(y))$
\end{enumerate}
\end{Definition}

In the following table we summarize what it means for a valued $C$-field $(K,v)$ to satisfy one of these sentences,
as well as the quantifier complexity for the case that $\phi(x)$ is an $\exists$-$\mathcal{L}_{\rm vf}(C)$-formula:

\begin{center}
\begin{tabular}{r|c|r|}
 & holds in $(K,v)$ iff & quantifiers \\ \hline
$({\rm SO}_\phi)$ & $\phi(K)\subseteq\mathcal{O}_v$ & $\forall$ \\ \hline
$({\rm CO}_\phi)$ & $\phi(K)\supseteq\mathcal{O}_v$ & $\forall\exists$ \\\hline
$({\rm SM}_\phi)$ & $\phi(K)\subseteq\mathfrak{m}_v$  & $\forall$ \\ \hline
$({\rm CM}_\phi)$ & $\phi(K)\supseteq\mathfrak{m}_v$ & $\forall\exists$ \\\hline
$({\rm IM}_\phi)$ & $\phi(K)\cap(\mathfrak{m}_v\setminus\{0\})\neq\emptyset$  & $\exists$\\ \hline
$({\rm R}_\phi)$ & ${\rm res}(\phi(K))=Kv$ & $\forall^k\exists$ \\ \hline
\end{tabular}
\end{center}

With these interpretations, the following lemma is obvious:

\begin{Lemma}\label{lem:relationships.between.important.sentences}
Let $\phi(x)$ be an $\mathcal{L}_{\rm vf}(C)$-formula and $(K,v)$ a nontrivially valued $C$-field.
\begin{enumerate}
\item $\phi(K)=\mathcal{O}_v$ $\Leftrightarrow$
 $(K,v)\models({\rm SO}_\phi)\wedge({\rm CO}_\phi)$
 $\Rightarrow$
 $(K,v)\models({\rm SO}_\phi)\wedge({\rm R}_\phi)\wedge({\rm IM}_\phi)$
\item $\phi(K)=\mathfrak{m}_v$ $\Leftrightarrow$ $(K,v)\models({\rm SM}_\phi)\wedge({\rm CM}_\phi)$
$\Rightarrow$  $(K,v)\models({\rm SM}_\phi)\wedge({\rm IM}_\phi)$
\end{enumerate}
\end{Lemma}

\begin{Remark}[cf.~{\cite[Remark 6.6]{AnscombeFehm}}]\label{rem:strategy}
The strategy now is to use the results of \cite{AnscombeFehm},
which allow us to `transfer' the truth of $\forall^k\exists$-$\mathcal{L}_{\rm vf}(C)$-sentences between equicharacteristic henselian nontrivially valued fields with the same residue field (see \ref{cor:transfer}).
According to the table above, if $\phi$ is an $\exists$-$\mathcal{L}_{\rm vf}(C)$-formula, these results can be applied directly to
$({\rm SO}_\phi)$ and $({\rm SM}_\phi)$, 
but unfortunately not to $({\rm CO}_\phi)$ and $({\rm CM}_\phi)$.
Therefore, we 
instead have to work with the weaker statements
$({\rm R}_\phi)$ and $({\rm IM}_\phi)$, 
to which the results can be applied,
and modify the formula $\phi$ suitably.
\end{Remark}

For $\mathcal{F}$ an elementary class of $C$-fields, we recall that $\mathcal{H}_e'(\mathcal{F})$ denotes the class of equicharacteristic henselian nontrivially valued $C$-fields with residue field in $\mathcal{F}$.
So far we have formulated our results (in particular \ref{prp:GDM.version.1}) in terms of the elementary class $\mathcal{H}_e'(\mathcal{F})$. However, for the rest of the section it will be convenient to make the following notational change: instead of $\mathcal{H}_e'(\mathcal{F})$, we formulate our results in terms of the theory $\mathbf{T}_{\mathcal{F}}$, which is defined as follows.

\begin{Definition}\label{def:T_F}
We let $\mathbf{T}_{\mathcal{F}}$ be the $\mathcal{L}_{\rm vf}(C)$-theory of equicharacteristic henselian nontrivially valued $C$-fields with residue field a member of $\mathcal{F}$. For a single $C$-field $F$, we let $\mathbf{T}_{F}$ denote the theory of equicharacteristic henselian nontrivially valued $C$-fields with residue field elementarily equivalent to $F$.
\end{Definition}

Thus $\mathbf{T}_{\mathcal{F}}$ is the theory of $\mathcal{H}_e'(\mathcal{F})$ and $\mathcal{H}_e'(\mathcal{F})$ is the class of models of $\mathbf{T}_{\mathcal{F}}$. The rest of this section is devoted to proving \autoref{prp:GDM.version.2}, which is simply \autoref{prp:GDM.version.1} rewritten using our new notation.

\subsection{Existential transfer principle}

First we recall the `transfer principle' from \cite{AnscombeFehm} which, as we have indicated, will be our main tool in this section. 
Let $F/C'$ be a separable field extension.
In \cite{AnscombeFehm}, $\mathbf{T}_{F/C'}$ denotes the theory of equicharacteristic henselian nontrivially valued fields $(K,v,d_c)_{c\in C'}$ in the language $\mathcal{L}_{\rm vf}(C')$
for which $c\mapsto d_c$ gives a homomorphism $C'\rightarrow K$,
the valuation $v$ is trivial on $D:=\{d_c:c\in C'\}$,
and $(Kv,d_cv)_{c\in C'}\equiv(F,c)_{c\in C'}$.
That is, $\mathbf{T}_{F/C'}$ is the same as $\mathbf{T}_{F}$ where we view $F$ as a $C'$-field.

\begin{Proposition}[{\cite[Corollary 5.6]{AnscombeFehm}}]\label{fact:transfer}
Let $\psi(\mathbf{x})$ be an $\exists$-$\mathcal{L}_{\rm vf}(C')$-formula with free variables $\mathbf{x}$ belonging to the residue field sort. 
Suppose there exists $(K,v)\models\mathbf{T}_{F/C'}\cup\{\forall^{k}\mathbf{x}\;\psi(\mathbf{x})\}$. 
Then there exists $n\in\mathbb{N}$ such that, for all $(L,w)\models\mathbf{T}_{F/C'}$, we have ${}^{\mathbf{x}}Lw\subseteq\psi(L^{p^{-n}})$, where $p\geq1$ is the characteristic exponent of $F$.
\end{Proposition}

\begin{Proposition}[{\cite[Corollary 5.7]{AnscombeFehm}}]\label{fact:transfer.perfect}
Suppose that $F$ is perfect. Then for any $\forall^{k}\exists$-$\mathcal{L}_{\rm vf}(C')$-sentence $\phi$, either $\mathbf{T}_{F/C'}\models\phi$ or $\mathbf{T}_{F/C'}\models\neg\phi$.
\end{Proposition}

For convenience, we collect these results together into one corollary which covers almost all of the applications and is expressed using the new notation.

\begin{Definition}
For a $C$-field $F$ let $C'=C_F$ and let
$\alpha:C\rightarrow C'$ denote the structure homomorphism of $F$.
If $(K,v)\models\mathbf{T}_{F/C'}$, then composing the structure homomorphism $C'\rightarrow K$
with $\alpha$ turns $K$ into a $C$-field $K^\circ$ with $(K^\circ,v)\models\mathbf{T}_F$.
To a $\mathcal{L}_{\rm vf}(C)$-formula $\phi(\mathbf{x})$ we assign a $\mathcal{L}_{\rm vf}(C')$-formula $\phi'(\mathbf{x})$
by applying $\alpha$ to all the constants.
\end{Definition}

\begin{Remark}\label{rem:translate}
Note that for $(K,v)\models\mathbf{T}_{F/C'}$, an $\mathcal{L}_{\rm vf}(C)$-formula $\phi(\mathbf{x})$
and $\mathbf{a}\in K^n$ we trivially have
$(K^\circ,v)\models\phi(\mathbf{a})$ if and only if $(K,v)\models\phi'(\mathbf{a})$.
\end{Remark}

\begin{Lemma}\label{lem:star}
Let $F$ be a $C$-field that satisfies $(*)$.
Then $C_F$ is perfect and if $(K,v)\models\mathbf{T}_F$, 
then $v$ is trivial on $C_K$ and the residue map induces an isomorphism $C_K\rightarrow C_F$.
\end{Lemma}

\begin{proof}
In case (a), $C_F$ and $C_K$ are algebraic extensions of the prime field; in case (b), $C\cong C_F\cong C_K$ is a perfect field.
In particular, $C_F$ is perfect in both cases.
In case (a), the assumption that $v$ is equicharacteristic implies that it is trivial on the prime field, hence also on $C_K$.
In case (b), since $v$ is nonnegative on the image $C_K$ of $C$, which is a field, it is trivial on $C_K$.
\end{proof}

\begin{Lemma}\label{rem:star}
Let $F$ be a $C$-field that satisfies $(*)$.
Then the map $(K,v)\mapsto (K^\circ,v)$ from the models of $\mathbf{T}_{F/C'}$ to the models of $\mathbf{T}_{F}$
has an inverse $(K,v)\mapsto (K',v)$,
and for $(K,v)\models\mathbf{T}_F$, an $\mathcal{L}_{\rm vf}(C)$-formula $\phi(\mathbf{x})$ and $\mathbf{a}\in K^n$ we have
$$
 (K,v)\models\phi(\mathbf{a}) \quad\Longleftrightarrow\quad (K',v)\models\phi'(\mathbf{a}).
$$ 
\end{Lemma}

\begin{proof}
If $(K,v)\models\mathbf{T}_F$, then \autoref{lem:star} gives an isomorphism $\beta:C_K\rightarrow C_F=C'$,
so $\beta^{-1}:C'\rightarrow K$ turns $K$ into a $C'$-field $K'$ with $(K',v)\models\mathbf{T}_{F/C'}$.
Obviously this is inverse to $(K,v)\mapsto (K^\circ,v)$.
The second claim is then immediate from \autoref{rem:translate}.
\end{proof}

\begin{Corollary}\label{cor:transfer}
Let $F$ be a $C$-field of characteristic exponent $p\geq 1$ that satisfies $(*)$.
\begin{enumerate}
\item Let $\psi(x)$ be an $\exists$-$\mathcal{L}_{\rm vf}(C)$-formula with free variable $x$ belonging to the residue field sort. 
If $\mathbf{T}_{F}\cup\{\forall^{k}x\;\psi(x)\}$ is consistent then there exists $n\in\mathbb{N}$ such that for all $(L,w)\models\mathbf{T}_{F}$ we have $(Lw)^{p^{n}}\subseteq\psi(L)$.
\item Let $\phi$ be an $\exists$-$\mathcal{L}_{\rm vf}(C)$-sentence. If $\mathbf{T}_{F}\cup\{\phi\}$ is consistent, then $\mathbf{T}_{F}\models\phi$.
\item Let $\phi$ be an $\forall$-$\mathcal{L}_{\rm vf}(C)$-sentence. If $\mathbf{T}_{F}\cup\{\phi\}$ is consistent, then $\mathbf{T}_{F}\models\phi$.
\end{enumerate}
\end{Corollary}

\begin{proof}
We first show (1).
Let again $C'=C_F$.
If $\mathbf{T}_{F}\cup\{\forall^{k}x\;\psi(x)\}$ is consistent, then by \ref{rem:star} also
$\mathbf{T}_{F/C'}\cup\{\forall^{k}x\;\psi'(x)\}$ is consistent, 
and it suffices to show that there exists $n\in\mathbb{N}$ such that 
$(Lw)^{p^{n}}\subseteq\psi'(L)$
for all $(L,w)\models\mathbf{T}_{F/C'}$.

If $F$ satisfies (b), then \ref{fact:transfer.perfect} shows that the claim holds for $n=0$.
Now suppose instead that $F$ satisfies (a), i.e.~$C$ is integral over its prime ring. 
There exists a quantifier-free $\mathcal{L}_{\rm vf}$-formula $q(x,\mathbf{y},\mathbf{z})$ such that $\psi'(x)$ is equivalent 
to $\exists\mathbf{y}\;q(x,\mathbf{y},\mathbf{c})$ for some finite tuple $\mathbf{c}\subseteq C'$ of parameters.
Let $(L,w)\models\mathbf{T}_{F/C'}$.
By \ref{fact:transfer}, there exists $n_{1}\in\mathbb{N}$ such that $Lw\subseteq\psi'(L^{p^{-n_{1}}})$, 
i.e.\ $Lw$ is contained in the projection onto the $x$-coordinate of the set of realizations $(a,\mathbf{b})$ of
$q(x,\mathbf{y},\mathbf{c})$ in $L^{p^{-n_{1}}}$,
where we view $C'\subseteq L\subseteq L^{p^{-n_1}}$.
There exists $m\in\mathbb{N}$ such that $\mathbf{c}$ is fixed pointwise by the $m$-th power of the Frobenius endomorphism, i.e. by the map $f:x\longmapsto x^{p^{m}}$. 
Note that the $n_{1}$-th iterated composition of the map $f$ with itself is the map $f^{n_{1}}:x\longmapsto x^{p^{mn_{1}}}$. 
Thus $\mathbf{c}=\mathbf{c}^{p^{mn_{1}}}$.

Now let $a\in\mathcal{O}_w$. Then $aw\in\psi'(L^{p^{-n_{1}}})$, so there exists $\mathbf{b}\subseteq L^{p^{-n_{1}}}$ such that
$$
 (L^{p^{-n_1}},w)\models q(aw,\mathbf{b},\mathbf{c}).
$$
Applying the endomorphism $x\longmapsto x^{p^{mn_{1}}}$ of $L^{p^{-n_1}}$, we get
$$
 (L^{p^{-n_1}},w)\models q(aw^{p^{mn_{1}}},\mathbf{b}^{p^{mn_{1}}},\mathbf{c}).
$$
Since $\mathbf{b}^{p^{mn_{1}}}=(\mathbf{b}^{p^{n_{1}}})^{p^{(m-1)n_{1}}}\subseteq L$ and $q$ is quantifier free, we have $aw^{p^{mn_{1}}}\in\psi'(L)$. Setting $n:=mn_{1}$ completes the proof of $(1)$.

Having shown $(1)$, for $(2)$ we simply view the $\exists$-sentence $\phi$ as an $\forall^{k}\exists$-formula. More precisely, let $x$ be a variable that does not appear in $\phi$. Then $\phi$ is logically equivalent to both $\forall^{k}x\;\psi(x)$ and $\exists^{k}x\;\psi(x)$, with $\psi(x)=\phi$. Our assumption can be restated as: $\mathbf{T}_{F}\cup\{\forall^{k}x\;\psi(x)\}$ is consistent. Now let $(L,w)\models\mathbf{T}_{F}$. Part $(1)$ implies that $(Lw)^{p^{n}}\subseteq\psi(L)$. In particular $\psi(L)$ is non-empty, i.e. $(L,w)\models\exists^{k}x\;\psi(x)$. Therefore $(L,w)\models\phi$ as required.

Finally, $(3)$ follows immediately from $(2)$, since if $\mathbf{T}_F\not\models\phi$, then
$\mathbf{T}_F\cup\{\neg\phi\}$ is consistent, so $\mathbf{T}_F\models\neg\phi$, hence $\mathbf{T}_F\cup\{\phi\}$ is inconsistent.
\end{proof}

\subsection{The valuation ideal}

As before, $\mathcal{F}$ denotes an elementary class of $C$-fields that satisfies $(*)$. 
In this subsection we focus on the property $({\rm CM}_{\phi})$. 
As remarked above,
we cannot simply apply \autoref{cor:transfer}
to conclude from $(K,v)\models({\rm CM}_{\phi})$ that $(L,w)\models({\rm CM}_{\phi})$
for all $(L,w)\models\mathbf{T}_F$.
Instead we make use of the weaker property $({\rm IM}_{\phi})$ and proceed as follows. 
First we show in \autoref{lem:embedding.t} that
$$
(F(t)^{h},v_{t})\models({\rm CM}_{\phi})\implies\mathbf{T}_{F}\models({\rm CM}_{\phi}).
$$
Then, for each $F\in\mathcal{F}$, we use \autoref{cor:transfer} to `transfer' $({\rm IM}_{\phi})$, so that $(F(t)^{h},v_{t})\models({\rm IM}_{\phi})$. This allows us to identify a formula $\phi_{n}$ such that $(F(t)^{h},v_{t})\models({\rm CM}_{\phi_{n}})$. Finally, we use a compactness argument to find $\phi_{n}$ uniformly across $\mathcal{F}$.

\begin{Lemma}\label{lem:embedding.t}
Let $F$ be a $C$-field satisfying $(*)$ and
let $\phi(x)$ be an $\exists$-$\mathcal{L}_{\rm vf}(C)$-formula
with free variable $x$ belonging to the field sort. 
Suppose that $(F(t)^{h},v_{t})\models\phi(t)\wedge\phi(0)$. 
Then $\mathbf{T}_{F}\models({\rm CM}_\phi)$.
\end{Lemma}

\begin{proof}
Let $(K,v)\models\mathbf{T}_{F}$. We aim to show that $(K,v)\models({\rm CM}_{\phi})$, i.e $\mathfrak{m}_{v}\subseteq\phi(K)$. 

By passing, if necessary, to an elementary extension, we may assume that $(K,v)$ is $|F|^{+}$-saturated. 
Since $F\equiv Kv$, there is an $\mathcal{L}_{\mathrm{ring}}(C)$-embedding $F\longrightarrow Kv$. 
By $(*)$, the residue map induces an isomorphism $C_K\rightarrow C_F$ (see \ref{lem:star}).
Let $f:C_F\rightarrow C_K$ denote its inverse.
By \cite[Lemma 2.3]{AnscombeFehm}, 
we can extend $f$ as a partial section of the residue map to any finitely generated subextension of $F/C_F$. 
By saturation, $f$ extends to an $\mathcal{L}_{\rm vf}(C)$-embedding $f:(F,v_{0})\longrightarrow(K,v)$, which is also a partial section of the residue map, where $v_{0}$ denotes the trivial valuation on $F$.

Let $s\in\mathfrak{m}_{v}\setminus\{0\}$ and note that $s$ is transcendental over $f(F)$. By sending $t\longmapsto s$, we may extend $f$ to an $\mathcal{L}_{\rm vf}(C)$-embedding $(F(t),v_{t})\longrightarrow(K,v)$. Since $(K,v)$ is henselian, we may again extend $f$ to an $\mathcal{L}_{\rm vf}(C)$-embedding $(F(t)^{h},v_{t})\longrightarrow(K,v)$. Since existential sentences `go up', we have that $s\in\phi(K)$, and also that $0\in\phi(K)$. This shows that $\mathfrak{m}_{v}\subseteq\phi(K)$.
\end{proof}

\begin{Lemma}\label{lem:henselian.approx}
Let $(K,v)$ be a henselian nontrivially valued field, let $E\subseteq K$ be a subfield, and let $b\in K$ be separably algebraic over $E$. 
Suppose that $(K,v)$ is $|E|^{+}$-saturated. 
Then there exists $b'\in K$ which is transcendental over $E$ such that $b\in E(b')^{h}$, the henselisation of $E(b')$ with respect to $v$.
\end{Lemma}

\begin{proof}
Let $v$ denote the unique extension of $v$ to $K^{\rm alg}$ and let
$$
 \gamma:=\max\{v(\sigma b-b): \sigma\in{\rm Gal}(E),\sigma b\neq b\} \in v E^{\rm alg}.
$$ 
Since $vE\subseteq vK$ is cofinal in $v E^{\rm alg}$, there exists $\gamma'\in vE$ such that $\gamma\leq\gamma'$. 
Then ${\rm B}_K(\gamma',b)=\{x\in K\;|\;v(x-b)\geq\gamma'\}$ is an infinite definable subset of $K$, 
so by saturation, there exists $b'\in {\rm B}_K(\gamma',b)$ which is transcendental over $E$. 
Then $b$ is also separably algebraic over $E(b')^h$, and
$$
 \max\{v(\sigma b -b):\sigma\in{\rm Gal}(E(b')^h) ,\sigma b\neq b\} \leq \gamma\leq\gamma'\leq v(b'-b),
$$
so Krasner's Lemma \cite[Theorem 4.1.7]{EP} implies that $b\in E(b')^h(b')=E(b')^h$.
\end{proof}

\begin{Lemma}\label{lem:relative.inseparable.closure.of.singleton.over.perfect.subfield}
Let $C'\subseteq F\subseteq K$ be a tower of fields of characteristic exponent $p$ and let $a\in K$. 
Suppose that $C'$ is perfect and $F/C'$ is finitely generated. 
Then there exists $m\in\mathbb{N}$ such that $a^{p^{-m}}\in F(a)$ and $F(a)$ is separable over $C'(a^{p^{-m}})$.
\end{Lemma}

\begin{proof}
Let $E_{a}$ (respectively, $E_{s}$) denote the relative algebraic (resp., separable algebraic) closure of $C'(a)$ in $F(a)$. 
Since $F/C'$ is finitely generated, so is $E_a/C'$,
hence \cite[2.7.2 and 2.7.5]{FriedJarden} shows that $F(a)/E_{a}$ is regular, in particular separable. 
By \cite[3.10.2]{Sti}, $E_{s}=E_{a}^{p^{m}}$, for some $m\geq 0$. 
Since $E_{s}$ is separable over $C'(a)$, by the Frobenius isomorphism we have that $E_{a}$ is separable over $C'(a)^{p^{-m}}=C'(a^{p^{-m}})$.
$$
\xymatrix{
 F\ar@{-}[rr]\ar@{-}[dd] & & F(a)\ar@{-}[d]^{\rm reg.} \\
 & E_s\ar@{-}[r]^{\rm p.i.}\ar@{-}[d]^{\rm sep.} & E_a\ar@{-}[d]^{\rm sep.} \\
 C'\ar@{-}[r] & C'(a)\ar@{-}[r]^{\rm p.i.} & C'(a^{p^{-m}}) \\
}
$$
Therefore $F(a)$ is separable over $C'(a^{p^{-m}})$.
\end{proof}

\begin{Definition}
For a field $K$, subsets $X,Y\subseteq K$ and $n\in\mathbb{N}$ we use the notations
\begin{eqnarray*}
 X^{(n)} &:=& \{ x^n : x\in X\},\\
 X\pm Y &:=& \{x\pm y:x\in X,y\in Y\}.
\end{eqnarray*}
\end{Definition}

\begin{Proposition}\label{prp:monkey.lemma} 
Let $F$ be a $C$-field that satisfies $(*)$
and let $(K,v)$ be an equicharacteristic henselian valued $C$-field with value group $vK=\mathbb{Z}$ and residue field $Kv=F$. 
For any $\exists$-$\mathcal{L}_{\rm vf}(C)$-formula $\phi(x)$ with free variable from the field sort and $(K,v)\models({\rm IM}_\phi)$
there exists $n\in\mathbb{N}$ such that
$$
 \mathfrak{m}_{v}^{(n)}\subseteq\phi(K)-\phi(K).
$$ 
\end{Proposition}

\begin{proof}
By the assumptions,
$\phi(x)$ is logically equivalent to $\exists\mathbf{y}\;q(x;\mathbf{y})$
with a quantifier-free $\mathcal{L}_{\rm vf}(C)$-formula $q(x;\mathbf{y})$,
and 
there exist $0\neq a\in\mathfrak{m}_{v}$ and a $\mathbf{y}$-tuple $\mathbf{b}\subset K$ such that
$(K,v)\models q(a;\mathbf{b})$.

Let $C'=C_F$ and note that $a$ is transcendental over $C'$, by $(*)$.
Since $C'$ is perfect, we may apply \autoref{lem:relative.inseparable.closure.of.singleton.over.perfect.subfield} 
to the tower $C'\subseteq C'(\mathbf{b})\subseteq K$ of field extensions and the element $a\in K$
to find $m\in\mathbb{N}$ such that $a':=a^{p^{-m}}\in C'(a,\mathbf{b})$ and $C'(a',\mathbf{b})$ is separable over $C'(a')$,
where $p$ denotes the characteristic exponent of $F$.
By separability and reordering the tuple $\mathbf{b}$ if necessary, we may write $\mathbf{b}=(\mathbf{b}_{1},\mathbf{b}_{2})$ such that $\mathbf{b}_{1}$ is algebraically independent over $C'(a')$ and $C'(a',\mathbf{b})$ is separably algebraic over $E:=C'(a',\mathbf{b}_{1})$. 
Let $b\in C'(a',\mathbf{b})$ be a primitive element of this extension, so that $C'(a',\mathbf{b})=E(b)$.

Applying \autoref{lem:henselian.approx} to $E$, there is an elementary extension
$(K,v)\preceq(K^{*},v^{*})$ and an element $b'\in K^{*}$ which is transcendental over $E$ such that $b\in E(b')^{h}$, 
where the henselisation is taken with respect to the restriction of $v^{*}$. 
Consequently $C'(a',\mathbf{b})\subseteq E(b')^{h}$.
$$
\xymatrix{
  D\ar@{-}[r]\ar@{-}[d] & E(b')\ar@{-}[r]\ar@{-}[d]^{\rm p.tr.} & E(b')^h\ar@{-}[r]\ar@{-}[d] & K^*\ar@{-}[d] \\
  C'(\mathbf{b}_1)\ar@{-}[r]\ar@{-}[d] & E\ar@{-}[r]^{\rm sep. alg.}\ar@{-}[d]^{\rm p.tr.} & E(b)\ar@{-}[r] & K \\
  C'\ar@{-}[r]^{\rm p.tr.} & C'(a') & &
}
$$
Let $D:=C'(\mathbf{b}_{1},b')$ and note that $a'$ is transcendental over $D$. 
Let $\Psi(x)$ be the quantifier-free $\mathcal{L}_{\rm vf}$-type of $a'$ over $D$,
i.e.\ the set of quantifier-free $\mathcal{L}_{\rm vf}(D)$-formulas $\psi(x)$ in one free variable $x$ belonging to the field sort
that satisfy $(K^*,v^*)\models \psi(a')$. 
If some $\hat{a}\in K^{*}$ realises $\Psi(x)$ then there is an $\mathcal{L}_{\rm vf}(D)$-isomorphism $D(a')\longrightarrow D(\hat{a})$ with $a'\longmapsto\hat{a}$. 
By the universal property of the henselisation, it extends to an $\mathcal{L}_{\rm vf}(D)$-isomorphism $D(a')^{h}\longrightarrow D(\hat{a})^{h}$. 
Since $C'(a',\mathbf{b})\subseteq E(b')^{h}=D(a')^{h}$ and $D(\hat{a})^{h}\subseteq K^{*}$, this implies $(K^{*},v^{*})\models\phi(\hat{a}^{p^{m}})$. Thus, by compactness, there is a single quantifier-free $\mathcal{L}_{\rm vf}(D)$-formula $\psi(x)\in\Psi(x)$ such that 
$$
 (K^{*},v^{*})\models\;\forall x\;(\psi(x)\longrightarrow\phi(x^{p^{m}})).
$$
By \ref{Lemma:qfdefinable}, $\psi(K^*)=A\cup U$ with $A$ algebraic over $D$ and $U$ open.
As $a'\in \psi(K^*)$ is transcendental over $D$, we get that
the set defined by $\phi(x^{p^m})$ in $K^*$ contains 
a ball ${\rm B}_{K^*}(\gamma,a')$
for some $\gamma\in v^*K^*$.
As $(K,v)\prec(K^*,v^*)$, the set defined by $\phi(x^{p^m})$ in $K$
therefore contains a ball ${\rm B}_K(l,a')$
for some $l\in vK=\mathbb{Z}$. 
Without loss of generality, $l\in\mathbb{N}$.

Finally, let $n:=lp^{m}$ and let $c\in\mathfrak{m}_{v}^{(n)}$.
Then 
$$
 c':=c^{p^{-m}}\in\mathfrak{m}_{v}^{(l)}\subseteq {\rm B}_K(l,a')-a',
$$
and so $(K,v)\models\phi((c'+a')^{p^m})$,
i.e.~$(K,v)\models\phi(c+a)$. 
Therefore $c\in\phi(K)-a\subseteq\phi(K)-\phi(K)$.
\end{proof}

\begin{Remark}\label{rem:phi_n}
In the setting of \autoref{prp:monkey.lemma},
if we denote by $\phi_n(x)$ the $\exists$-$\mathcal{L}_{\rm vf}(C)$-formula
$$
 \exists y \exists z\; (x^n=y-z\wedge\phi(y)\wedge\phi(z)),
$$
then
$$
 (K,v)\models({\rm CM}_{\phi_{n}}).
$$ 
Furthermore, $\phi_n(x)$ defines the set of $n$-th roots of the set of differences between elements in the set defined by $\phi(x)$. 
Thus if $\phi(x)$ defines a subset of $\mathcal{O}_{v}$ (respectively, $\mathfrak{m}_{v}$) then so does $\phi_{n}(x)$, 
i.e. for all $n\in\mathbb{N}$ the theory of valued fields entails
$$
 ({\rm SO}_{\phi})\longrightarrow({\rm SO}_{\phi_{n}})
$$
and
$$
 ({\rm SM}_{\phi})\longrightarrow({\rm SM}_{\phi_{n}}).
$$
This fact is used several times in the proof of \autoref{prp:GDM.version.2}.
\end{Remark}

\subsection{Making a definition uniform}

In the rest of this section we allow ourselves to write $({\rm CM}_{\phi(x)})$ (rather than simply $({\rm CM}_{\phi})$), when needed, to highlight the r\^{o}le of the variable $x$ and to allow us the flexibility to substitute other terms in its place.

\begin{Lemma}\label{lem:bounded.roots.compactness}
Let $\mathcal{F}$ be an elementary class of $C$-fields.
Let $\phi(x)$ be an $\exists$-$\mathcal{L}_{\rm vf}(C)$-formula
with free variable $x$ belonging to the field sort. 
Suppose that for all $F\in\mathcal{F}$ there exists $n_{F}\in\mathbb{N}$ such that for all $(K,v)\models\mathbf{T}_{F}$ we have
$$
 \mathfrak{m}_{v}^{(n_{F})}\subseteq\phi(K)\qquad\left(\text{resp. }\mathcal{O}_{v}^{(n_{F})}\subseteq\phi(K)\right).
$$
Then there exists $n\in\mathbb{N}$ such that for all $(K,v)\models\mathbf{T}_{\mathcal{F}}$ we have
$$
 \mathfrak{m}_{v}^{(n)}\subseteq\phi(K)\qquad\left(\text{resp. }\mathcal{O}_{v}^{(n)}\subseteq\phi(K)\right).
$$
\end{Lemma}

\begin{proof}
This is a straightforward compactness argument. We first note that, for a valued field $(K,v)$, we have $\mathfrak{m}_{v}^{(n)}\subseteq\phi(K)$ if and only if $(K,v)\models({\rm CM}_{\phi(x^{n})})$. Next we consider the $\mathcal{L}_{\rm vf}(C)$-theory
$$
 \mathbf{T}_{\rm CM}:=\left\{\neg({\rm CM}_{\phi(x^{n})})\;\Big|\;n\in\mathbb{N}\right\}.
$$
Now let $(K,v)\models\mathbf{T}_{\mathcal{F}}$. Then $F:=Kv\in\mathcal{F}$ and, by assumption, we have $\mathfrak{m}_{v}^{(n_{F})}\subseteq\phi(K)$, i.e. $(K,v)\models({\rm CM}_{\phi(x^{n_{F}})})$. In particular we have $(K,v)\not\models\mathbf{T}_{\rm CM}$. This shows that $\mathbf{T}_{\mathcal{F}}\cup\mathbf{T}_{\rm CM}$ is inconsistent. 
By compactness, $\mathbf{T}_{\rm CM}$ is finitely inconsistent; thus there exists $M\in\mathbb{N}$ such that
$$\mathbf{T}_{\mathcal{F}}\models\bigvee_{m<M}({\rm CM}_{\phi(x^{m})}).$$
Finally, we let $n:=M!$. 
For $(K,v)\models\mathbf{T}_{\mathcal{F}}$ there exists $m<M$ such that $(K,v)\models({\rm CM}_{\phi(x^{m})})$,
so since $m|n$, we have that
$$
 \mathfrak{m}_{v}^{(n)}\subseteq\mathfrak{m}_{v}^{(m)}\subseteq\phi(K),
$$
as required to complete the proof of the first statement.

The second statement (i.e. the `respectively' version) is entirely analogous and the proof is exactly the same except we exchange $\mathfrak{m}$ for $\mathcal{O}$, and $({\rm CM})$ for $({\rm CO})$, etc.
\end{proof}

Now we prove the main result of this section, \autoref{prp:GDM.version.2}, by repeatedly using \autoref{cor:transfer} to `transfer' several key properties of $\phi$ which were introduced in \autoref{def:important.sentences}. We will frequently use \autoref{lem:relationships.between.important.sentences} without further comment. As noted earlier, the following is simply a rephrasing of \autoref{prp:GDM.version.1} using different notation.

\begin{Proposition}\label{prp:GDM.version.2}
Let $\mathcal{F}$ be an elementary class of $C$-fields that satisfies $(*)$. 
Let $\phi(x)$ be an $\exists$-$\mathcal{L}_{\mathrm{ring}}(C)$-formula.
\begin{enumerate}
\item {\bf ($\mathcal{O}$-case)} If for all $F\in\mathcal{F}$ the theory $\mathbf{T}_{F}\cup\{({\rm SO}_{\phi}),({\rm CO}_{\phi})\}$ is consistent, then there is an $\exists$-$\mathcal{L}_{\mathrm{ring}}(C)$-formula $\psi(x)$ such that
$$
 \mathbf{T}_{\mathcal{F}}\models({\rm SO}_{\psi})\wedge({\rm CO}_{\psi}).
$$
\item {\bf ($\mathfrak{m}$-case)} If for all $F\in\mathcal{F}$ the theory $\mathbf{T}_{F}\cup\{({\rm SM}_{\phi}),({\rm CM}_{\phi})\}$ is consistent, then there is an $\exists$-$\mathcal{L}_{\mathrm{ring}}(C)$-formula $\chi(x)$ such that
$$
 \mathbf{T}_{\mathcal{F}}\models({\rm SM}_{\chi})\wedge({\rm CM}_{\chi}).
$$
\end{enumerate}
\end{Proposition}

\begin{proof}
\bf Step 1. \rm For the moment we work in both cases. Our aim is to find a formula $\chi(x)$ (by a simple adaptation of the formula $\phi(x)$) such that $\mathbf{T}_{\mathcal{F}}\models({\rm CM}_{\chi})$.

Let $F\in\mathcal{F}$. Recall the $\exists$-$\mathcal{L}_{\rm vf}(C)$-sentence $({\rm IM}_{\phi})$. 
In each case our assumption entails that $\mathbf{T}_{F}\cup\{({\rm IM}_{\phi})\}$ is consistent,
hence
$\mathbf{T}_{F}\models({\rm IM}_{\phi})$ by \autoref{cor:transfer}(2).
In particular we have $(F(t)^{h},v_{t})\models({\rm IM}_{\phi})$. 
By applying \autoref{prp:monkey.lemma} we find $n_{F}\in\mathbb{N}$ such that
$$
 \mathfrak{m}_{v_{t}}^{(n_{F})}\subseteq\Phi(F(t)^{h}),
$$
where $\Phi(x)$ denotes the $\exists$-$\mathcal{L}_{\rm ring}(C)$-formula 
$$
 \exists y\exists z\left(x=y-z\wedge\phi(y)\wedge\phi(z)\right).
$$ 
In particular, we have
$$
 (F(t)^{h},v_{t})\models\Phi(t^{n_{F}})\wedge\Phi(0^{n_{F}}).
$$
Therefore, by \autoref{lem:embedding.t},
$$
 \mathbf{T}_{F}\models\left({\rm CM}_{\Phi(x^{n_{F}})}\right),
$$
i.e.~for all $(K,v)\models\mathbf{T}_{F}$ we have
$$
 \mathfrak{m}_{v}^{(n_{F})}\subseteq\Phi(K).
$$
Next, by applying \autoref{lem:bounded.roots.compactness}, 
there exists $n\in\mathbb{N}$ such that for all $(K,v)\models\mathbf{T}_{\mathcal{F}}$ we have
$$\mathfrak{m}_{v}^{(n)}\subseteq\Phi(K).$$
Finally, we let $\chi(x)$ be the $\exists$-$\mathcal{L}_{\rm ring}(C)$-formula $\Phi(x^{n})$ and rewrite the previous statement as
$$\mathbf{T}_{\mathcal{F}}\models\left({\rm CM}_{\chi}\right),$$
which is the desired conclusion.

\bf Step 2 ($\mathfrak{m}$-case). \rm Let $F\in\mathcal{F}$. 
Consider the $\forall$-$\mathcal{L}_{\rm vf}(C)$-sentence $({\rm SM}_{\phi})$. 
Trivially, our assumption entails that $\mathbf{T}_{F}\cup\{({\rm SM}_{\phi})\}$ is consistent,
hence
$\mathbf{T}_{F}\models({\rm SM}_{\phi})$ by \autoref{cor:transfer}(3).
Since every model of $\mathbf{T}_{\mathcal{F}}$ is a model of $\mathbf{T}_{F}$ for some $F\in\mathcal{F}$, we deduce that $\mathbf{T}_{\mathcal{F}}\models({\rm SM}_{\phi})$.

As noted in \autoref{rem:phi_n}, the theory of valued fields entails
$({\rm SM}_{\phi})\longrightarrow({\rm SM}_{\chi})$.
Therefore $\mathbf{T}_{\mathcal{F}}\models({\rm SM}_{\chi})$. This completes the ($\mathfrak{m}$-case).

\bf Step 2 ($\mathcal{O}$-case). \rm To begin with, this step is similar to the ($\mathfrak{m}$-case). Let $F\in\mathcal{F}$. 
Consider the $\forall$-$\mathcal{L}_{\rm vf}(C)$-sentence $({\rm SO}_{\phi})$. Trivially, our assumption entails that $\mathbf{T}_{F}\cup\{({\rm SO}_{\phi})\}$ is consistent,
hence
$\mathbf{T}_{F}\models({\rm SO}_{\phi})$ by \autoref{cor:transfer}(3).
Since every model of $\mathbf{T}_{\mathcal{F}}$ is a model of $\mathbf{T}_{F}$ for some $F\in\mathcal{F}$, we deduce that $\mathbf{T}_{\mathcal{F}}\models({\rm SO}_{\phi})$.
As noted in \autoref{rem:phi_n}, the theory of valued fields entails
$({\rm SO}_{\phi})\longrightarrow({\rm SO}_{\chi})$,
hence $\mathbf{T}_{\mathcal{F}}\models({\rm SO}_{\chi})$. Combining this with Step 1, we have that
$$
 \mathbf{T}_{\mathcal{F}}\models({\rm SO}_{\chi})\wedge({\rm CM}_{\chi}).
$$
Let $F\in\mathcal{F}$ and let $p\geq1$ be the characteristic exponent of $F$. Consider the $\forall^{k}\exists$-$\mathcal{L}_{\rm vf}(C)$-sentence $({\rm R}_{\phi})$. Our assumption entails that $\mathbf{T}_{F}\cup\{({\rm R}_{\phi})\}$ is consistent. 
By \autoref{cor:transfer}(1) there exists $m_{F}\in\mathbb{N}$ such that for all $(K,v)\models\mathbf{T}_{F}$ we have
$$(Kv)^{(p^{m_{F}})}\subseteq\phi(K)v.$$

Let $\psi'(x)$ be the $\exists$-$\mathcal{L}_{\rm ring}(C)$-formula
$$
 \exists y\exists z\;(x=y+z\wedge\phi(y)\wedge\chi(z)).
$$
We will show that $\mathcal{O}_{v}^{(p^{m_{F}})}\subseteq\psi'(K)$ for all $(K,v)\models\mathbf{T}_{F}$. Let $a\in\mathcal{O}_{v}$. Since $(Kv)^{(p^{m_{F}})}\subseteq\phi(K)v$, there exists $b\in\phi(K)$ such that $(av)^{p^{m_{F}}}=bv$. Let $c:=a^{p^{m_{F}}}-b$. Then $c\in\mathfrak{m}_{v}\subseteq\chi(K)$. Thus $a^{p^{m_{F}}}=b+c\in\phi(K)+\chi(K)=\psi'(K)$, as required.

By applying \autoref{lem:bounded.roots.compactness}, 
there exists $m\in\mathbb{N}$ such that
$$\mathcal{O}_{v}^{(m)}\subseteq\psi'(K),$$
for all $(K,v)\models\mathbf{T}_{\mathcal{F}}$.

If we let $\psi(x)$ be the formula $\psi'(x^{m})$ then we may rewrite the previous statement as
$\mathbf{T}_{\mathcal{F}}\models({\rm CO}_{\psi})$.
Finally, since valuation rings are integrally closed and we have already seen that $\mathbf{T}_{\mathcal{F}}\models({\rm SO}_{\phi})\wedge({\rm SO}_{\chi})$, we deduce that $\mathbf{T}_{\mathcal{F}}\models({\rm SO}_{\psi})$, as required.
\end{proof}

For a class $\mathcal{C}$ of $C$-fields or valued $C$-fields we denote by $U\mathcal{C}$ the closure of $\mathcal{C}$ under ultraproducts
and by $E\mathcal{C}$ the closure of $\mathcal{C}$ under elementary equivalence,
and in the case of valued $C$-fields we let $\mathcal{C}v:=\{Kv:(K,v)\in\mathcal{C}\}$.

\begin{Corollary}\label{Cor:KtoHeF}
Let $\mathcal{K}$ be a class of equicharacteristic henselian nontrivially valued $C$-fields and let $\mathcal{F}$ be the smallest elementary class of $C$-fields that contains $\mathcal{K}v$. Suppose that $\mathcal{F}$ satisfies $(*)$.
Let $\phi(x)$ be an $\exists$-$\mathcal{L}_{\mathrm{ring}}(C)$-formula.
\begin{enumerate}
\item If $\phi(K)=\mathcal{O}_v$ for all $(K,v)\in\mathcal{K}$, 
then there exists an $\exists$-$\mathcal{L}_{\mathrm{ring}}(C)$-formula $\psi(x)$ with
$\psi(K)=\mathcal{O}_v$ for all $(K,v)\in\mathcal{H}_e'(\mathcal{F})$.
\item If $\phi(K)=\mathfrak{m}_v$ for all $(K,v)\in\mathcal{K}$, 
then there exists an $\exists$-$\mathcal{L}_{\mathrm{ring}}(C)$-formula $\chi(x)$ with
$\chi(K)=\mathfrak{m}_v$ for all $(K,v)\in\mathcal{H}_e'(\mathcal{F})$.
\end{enumerate}
\end{Corollary}

\begin{proof}
We show how to deduce (1) from \autoref{prp:GDM.version.2}(1). 
The proof of (2) from \autoref{prp:GDM.version.2}(2) is completely analogous.

Let $\mathcal{K}^{\dagger}$ be the class of equicharacteristic henselian nontrivially valued $C$-fields $(K,v)$ for which
\begin{equation*}\label{eqn:1}
(\dagger)\qquad\mathcal{O}_{v}=\phi(K)
\end{equation*}
holds, and let $\mathcal{F}^{\dagger}$ denote the smallest elementary class of $C$-fields containing $\mathcal{K}^{\dagger}v$. 
Then $\mathcal{K}\subseteq\mathcal{K}^{\dagger}$ and $\mathcal{F}\subseteq\mathcal{F}^{\dagger}=EU(\mathcal{K}^{\dagger}v)$, cf.~\cite[Exercise 4.1.18]{CK}.
Note that $\mathcal{K}^\dagger$ is closed under ultraproducts (i.e. $U\mathcal{K}^\dagger=\mathcal{K}^\dagger$) since $(\dagger)$ is an
$\mathcal{L}_{\rm vf}(C)$-elementary property and $\mathcal{K}^{\dagger}$ is an $\mathcal{L}_{\rm vf}(C)$-elementary class.
Thus,
$$
\mathcal{F}\subseteq\mathcal{F}^{\dagger}=EU(\mathcal{K}^{\dagger}v)\stackrel{\autoref{lem:residue.ultraproduct}}{=}E((U\mathcal{K}^{\dagger})v)=E(\mathcal{K}^{\dagger}v).
$$
That is, for every $F\in\mathcal{F}$ there exists $(K,v)\in\mathcal{K}^{\dagger}$ such that $F\equiv Kv$. 
Then $(K,v)\models\mathbf{T}_{F}\cup\{{\rm(SO_{\phi}))},{\rm(CO_{\phi})}\}$,
so the hypotheses of \autoref{prp:GDM.version.2}(1) are satisfied, the conclusion of which is $(1)$. 
\end{proof}

\section{Main Theorem}
\label{section:main.theorem}

\noindent
In this section we bring together the results developed so far into one main theorem
and draw some first conclusions. 
We state both cases (i.e. $\exists$- and $\forall$-definitions of the valuation ring) simultaneously.
Recall that $\forall$-definitions of the valuation ring correspond to $\exists$-definitions of the valuation ideal, cf.~\ref{rem:EAOm}.

\begin{Theorem}[Main Theorem]\label{thm:main}
Let $\mathcal{K}$ be a class of equicharacteristic henselian nontrivially valued $C$-fields,
let $\mathcal{F}$ be the smallest elementary class of $C$-fields that contains $\mathcal{K}v$,
and suppose that
\renewcommand{\theequation}{$*$}
\begin{equation}\label{star}
\begin{minipage}{11cm}
{\it \begin{enumerate} 
      \item[(a)] $C$ is integral over its prime ring, \underline{or} 
      \item[(b)] $C$ is a perfect field and every $F\in\mathcal{F}$ is perfect.
     \end{enumerate}
}
\end{minipage}
\end{equation}
Then for $Q\in\{\exists,\forall\}$ the following holds:
\begin{enumerate}
\item[(i)] The following properties are equivalent:
\begin{enumerate}
\item[$(0^{Q}_{e})$] The valuation ring is uniformly $Q$-$C$-definable in $\mathcal{K}$.
\item[$(1^{'Q}_{e})$] The valuation ring is uniformly $Q$-$C$-definable in $\mathcal{H}^{'}_{e}(\mathcal{F})$.
\item[$(1^{Q}_{e})$] The valuation ring is uniformly $Q$-$C$-definable in $\mathcal{H}_{e}(\mathcal{F})$.
\item[$(2^{Q}_{e})$] $\mathcal{F}$ does not have equicharacteristic embedded residue, if $Q=\exists$ (resp.~is not equicharacteristic large, if $Q=\forall$).
\end{enumerate}
\item[(ii)] Also the following properties are equivalent:
\begin{enumerate}
\item[$(1^{'Q})$] The valuation ring is uniformly $Q$-$C$-definable in $\mathcal{H}^{'}(\mathcal{F})$.
\item[$(1^{Q})$] The valuation ring is uniformly $Q$-$C$-definable in $\mathcal{H}(\mathcal{F})$.
\item[$(2^{Q})$] $\mathcal{F}$ does not have embedded residue, if $Q=\exists$ (resp.~is not large, if $Q=\forall$).
\end{enumerate}
\item[(iii)] Moreover, if $\mathcal{F}$ is unmixed, i.e.\ does not both contain fields of characteristic zero and fields of positive characteristic,
then all seven conditions
$(0^{Q}_{e})$, $(1^{'Q}_{e})$, $(1^{'Q})$, $(1^{Q}_{e})$, $(1^{Q})$, $(2^{Q}_{e})$, and $(2^{Q})$ are equivalent.
\end{enumerate}
\end{Theorem}

\begin{proof}
The strategy of proof is summarized in the following diagram:
$$
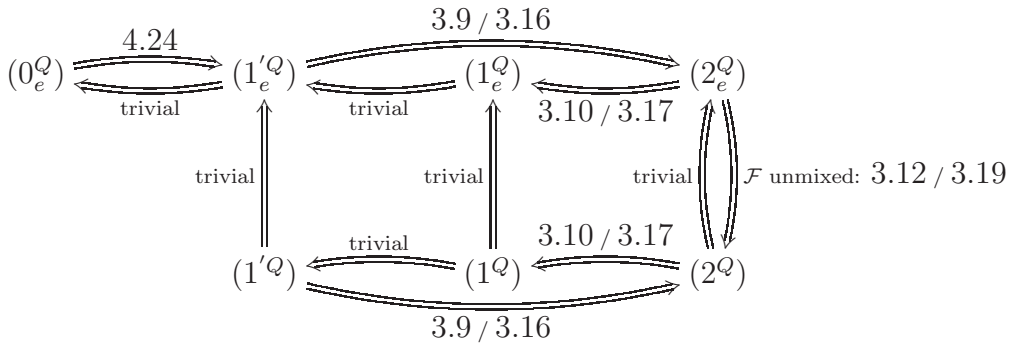
\xymatrix{
(0^{Q}_{e})
\ar@/^0.5pc/@{=>}[rr]^{\ref{Cor:KtoHeF}}
\ar@/_0.5pc/@{<=}[rr]_{\text{trivial}}
& &
(1^{'Q}_{e})
\ar@/^1.0pc/@{=>}[rrrr]^{\ref{lem:ER.implies.not.uniform}\;/\;\ref{lem:large.implies.not.uniform}}
\ar@/_0.5pc/@{<=}[rr]_{\text{trivial}}
\ar@/^0.0pc/@{<=}[dd]_{\text{trivial}}
& &
(1^{Q}_{e})
\ar@/_0.5pc/@{<=}[rr]_{\ref{lem:no.ER.implies.uniform}\;/\;\ref{lem:not.large.implies.uniform}}
\ar@/^0.0pc/@{<=}[dd]_{\text{trivial}}
& &
(2^{Q}_{e})
\ar@/_0.5pc/@{<=}[dd]_{\text{trivial}}
\ar@/^0.5pc/@{=>}[dd]^{\mathcal{F}\text{ unmixed: }
\ref{cor:existential.characterisation.one.characteristic}\;/\;\ref{cor:universal.characterisation.one.characteristic}}
\\
\\
& &
(1^{'Q})
\ar@/_1.0pc/@{=>}[rrrr]_{\ref{lem:ER.implies.not.uniform}\;/\;\ref{lem:large.implies.not.uniform}}
\ar@/^0.5pc/@{<=}[rr]^{\text{trivial}}
& &
(1^{Q})
\ar@/^0.5pc/@{<=}[rr]^{\ref{lem:no.ER.implies.uniform}\;/\;\ref{lem:not.large.implies.uniform}}
& &
(2^{Q})
}
$$
Since $\mathcal{K}\subseteq\mathcal{H}^{'}_{e}(\mathcal{F})\subseteq\mathcal{H}_{e}(\mathcal{F})\subseteq\mathcal{H}(\mathcal{F})$
and $\mathcal{H}^{'}_{e}(\mathcal{F})\subseteq\mathcal{H}'(\mathcal{F})\subseteq\mathcal{H}(\mathcal{F})$, we have five of the trivial implications
$(1^{'Q}_{e})\Longrightarrow(0^{Q}_{e})$,  $(1^{Q}_{e})\Longrightarrow(1^{'Q}_{e})$,
$(1^{Q})\Longrightarrow(1^{Q}_{e})$,
$(1^{'Q})\Longrightarrow(1^{'Q}_{e})$, and $(1^{Q})\Longrightarrow(1^{'Q})$.
Also $(2^{Q})\Longrightarrow(2^{Q}_{e})$ is immediate from the definition (cf.~\ref{rem:unmixed}).
By applying \ref{Cor:KtoHeF} we have the implication
$(0^{Q}_{e})\Longrightarrow(1^{'Q}_{e})$.

For $Q=\exists$ (resp.~$Q=\forall$),
both $(1^{'Q}_{e})\implies(2^{Q}_{e})$ and
$(1^{'Q})\implies(2^{Q})$
follow from
\ref{lem:ER.implies.not.uniform}
(resp.~\ref{lem:large.implies.not.uniform});
and both
$(2^{Q}_{e})\implies(1^{Q}_{e})$ and
$(2^{Q})\implies(1^{Q})$
follows from
\ref{lem:no.ER.implies.uniform}
(resp.~\ref{lem:not.large.implies.uniform}).
Finally,
by \ref{cor:existential.characterisation.one.characteristic}
(resp.~\ref{cor:universal.characterisation.one.characteristic}), if $\mathcal{F}$ is unmixed then we have $(2^{Q}_{e})\Longrightarrow(2^{Q})$.
\end{proof}


\begin{Remark}
Note that the theorem can also be applied if we start with an elementary class of $C$-fields $\mathcal{F}$,
as we can always find a suitable class of valued $C$-fields $\mathcal{K}$ as in the theorem,
e.g.~$\mathcal{K}=\mathcal{H}^{'}_{e}(\mathcal{F})$.
\end{Remark}

\subsection{The case of a single $C$-field $F$}

If we restrict our attention to henselian nontrivially valued fields with residue field elementarily equivalent to a given $F$, the situation becomes particularly nice. The following corollary is immediate from \autoref{thm:main}.

\begin{Corollary}\label{cor:single.F}
Let $F$ be any $C$-field that satisfies $(*)$.\\
The following are equivalent.
\begin{enumerate}
\item[$(0^{\exists}_{e})$] The valuation ring is $\exists$-$C$-definable in some $(K,v)\in\mathcal{H}_{e}'(F)$.
\item[$(1^{\exists})$] The valuation ring is uniformly $\exists$-$C$-definable in $\mathcal{H}(F)$.
\item[$(2^{\exists})$] $F$ does not have embedded residue.
\end{enumerate}
Also the following are equivalent.
\begin{enumerate}
\item[$(0^{\forall}_{e})$] The valuation ring is $\forall$-$C$-definable in some $(K,v)\in\mathcal{H}_{e}'(F)$.
\item[$(1^{\forall})$] The valuation ring is uniformly $\forall$-$C$-definable in $\mathcal{H}(F)$.
\item[$(2^{\forall})$] $F$ is not large. 
\end{enumerate}
\end{Corollary}

\begin{proof}
Let $\mathcal{F}$ be the class of $C$-fields elementarily equivalent to $F$. As explained in \ref{rem:unmixed}, $\mathcal{F}$ is unmixed. The result is immediate from \ref{thm:main}.
\end{proof}

Note that Theorem \autoref{thm:introthm} from the introduction follows immediately from \ref{cor:single.F} and \ref{lemma:Schoutens}
in the special case $C=\mathbb{Z}$.
See also \ref{rem:EAOm}.

\subsection{The necessity of the assumptions}

The assumption `unmixed' cannot simply be removed from part (iii) of \autoref{thm:main}, as the following example shows.

\begin{Example}\label{eg:unmixed.necessary.E}
In this example we set $Q:=\exists$ and $C:=\mathbb{Z}$; that is, we are studying the definability of valuation rings by existential formulas without parameters. Let $p$ be a fixed prime. 
Trivially, $\mathbb{F}_p$ does not have embedded residue, see also \autoref{prop:positive.applications}(1) below.
Later, in \ref{prop:positive.applications}(6) we will show that also $\mathbb{Q}$ does not have embedded residue,
i.e.~$\mathcal{F}_0:=\{F:F\equiv\mathbb{Q}\}$ does not have embedded residue.
So, since $\mathbb{F}_p$ and $\mathbb{Q}$ have different characteristics,
the elementary class $\mathcal{F}=\{\mathbb{F}_p\}\cup\mathcal{F}_0$ does not have equicharacteristic embedded residue.
On the other hand, the $p$-adic valuation on $\mathbb{Q}$ has residue field $\mathbb{F}_p$,
so $\mathcal{F}$ does have embedded residue.
Therefore, $\mathcal{F}$ satisfies $(2_e^\exists)$ but not $(2^\exists)$.
\end{Example}

Also the assumption in \autoref{thm:main} that the valued fields in $\mathcal{K}$ are `equicharacteristic' cannot be omitted, 
as the following example shows:

\begin{Example}
Again we let $Q:=\exists$ and $C:=\mathbb{Z}$.
Fix a prime $p>2$ and let $F:=\mathbb{F}_p^{\rm alg}$ be an algebraic closure of $\mathbb{F}_p$.
By \ref{lem:sep.cl}, $F$ has embedded residue,
so the valuation ring is not $\exists$-$\emptyset$-definable in any $(K,v)\in\mathcal{H}_e'(F)$, by \ref{cor:single.F}.
However, it is $\exists$-$\emptyset$-definable in some $(K,v)\in\mathcal{H}'(F)$:
For example, take the maximal unramified extension $K:=\mathbb{Q}_p^{\rm ur}$ of $\mathbb{Q}_p$ with the extension $v_p$ of the $p$-adic valuation.
Then $Kv_{p}=F$ and Julia Robinsons's formula $\exists y\;(y^2=1+px^2)$, which defines the valuation ring in $(\mathbb{Q}_p,v_{p})$ (see the introduction),
also defines the valuation ring in $(K,v_{p})$.
That is, for $\mathcal{K}=\{(K,v_p)\}$, $(0_e^Q)$ holds but $(1_e^Q)$ does not hold.
\end{Example}

\subsection{Mixed classes and families of local fields}

We now focus on a special case in which the `unmixed' assumption in part (iii) of \ref{thm:main} can be removed,
and we apply this to draw a conclusion on uniform existential definability of valuation rings in local fields,
continuing the theme of results developed in \cite{Cluckersetal} and \cite{Fehm}.

We work in the existential case, i.e. $Q=\exists$. 
Let $\mathcal{F}$ be an elementary class of $C$-fields.
If $\mathcal{F}$ is assumed to be unmixed then claim (iii) of \ref{thm:main} shows that all seven conditions (i.e. $(0^{\exists}_{e})$, $(1^{'\exists}_{e})$, $(1^{'\exists})$, $(1^{\exists}_{e})$, $(1^{\exists})$, $(2^{\exists}_{e})$, and $(2^{\exists})$) are equivalent. 
On the other hand, \ref{eg:unmixed.necessary.E} shows that the seven conditions may be inequivalent if we do not assume that $\mathcal{F}$ is unmixed.
However, if $\mathcal{F}$ is an elementary class of PAC fields then all seven conditions are necessarily equivalent,
as the following observation shows:

\begin{Proposition}\label{lem:PAC.nearly.unmixed}
Let $\mathcal{F}$ be an elementary class of finite or PAC $C$-fields. 
If $\mathcal{F}$ has embedded residue, then some $F\in\mathcal{F}$ has embedded residue.
\end{Proposition}

\begin{proof}
Suppose that $\mathcal{F}$ has embedded residue,
i.e.~there are $F_{1},F_{2}\in\mathcal{F}$ and a nontrivial valuation $v$ on $F_{1}$ such that $F_{1}v$ embeds into $F_{2}$. 
As $v$ is nontrivial, $F_1$ is not finite, hence $F_1$ is PAC,
and so $F_1v$ is an algebraically closed $C$-field, cf.~\ref{lem:residue.of.sep.closed} and \cite[Corollary 11.5.5]{FriedJarden}.
So since $F_1v$ embeds into $F_2$, \autoref{lem:sep.cl} shows that $F_{2}$ has embedded residue. 
\end{proof}

\begin{Corollary}\label{cor:QpZp}
Let $P$ be a set of prime numbers. The following are equivalent.
\begin{enumerate}
\item There exists an $\exists$-$\mathcal{L}_{\mathrm{ring}}$-formula $\phi(x)$ such that $\phi(\mathbb{Q}_{p})=\mathbb{Z}_{p}$ for all $p\in P$.
\item There exists an $\exists$-$\mathcal{L}_{\mathrm{ring}}$-formula $\phi(x)$ such that $\phi(\mathbb{F}_{p}((t)))=\mathbb{F}_{p}[[t]]$ for all $p\in P$.
\item There exists an $\exists$-$\mathcal{L}_{\mathrm{ring}}$-formula $\phi(x)$ such that $\phi(\mathbb{Q}_{p})=\mathbb{Z}_{p}$ and $\phi(\mathbb{F}_{p}((t)))=\mathbb{F}_{p}[[t]]$ for all $p\in P$.
\end{enumerate}
\end{Corollary}

\begin{proof}
Clearly $(3)\Longrightarrow(1),(2)$.

$(1)\Longrightarrow (2)$: By the Ax-Kochen transfer principle \cite{AxKochenI}, $\phi(\mathbb{F}_{p}((t)))=\mathbb{F}_{p}[[t]]$ 
for $p\in P_0$, with $P_0$ cofinite in $P$.
By \ref{thm:main}(i), the smallest elementary class $\mathcal{F}_0$ containing all $\mathbb{F}_p$, $p\in P_0$, does not have equicharacteristic embedded residue.
Therefore, also the elementary class $\mathcal{F}:=\mathcal{F}_0\cup\{\mathbb{F}_p:p\in P\setminus P_0\}$ does not have equicharacteristic embedded residue, so \ref{thm:main}(i) shows that (2) holds.

$(2)\Longrightarrow (3)$:
By \ref{thm:main}(i), the smallest elementary class $\mathcal{F}$ containing all $\mathbb{F}_p$, $p\in P$, does not have equicharacteristic embedded residue. Since it consists of fields that are finite or pseudofinite (hence PAC, by \cite[Corollary 11.3.4]{FriedJarden}), \ref{lem:PAC.nearly.unmixed} gives that $\mathcal{F}$ does not have embedded residue.
Therefore, \ref{thm:main}(ii) shows that (3) holds.
\end{proof}

\begin{Remark}
Note that by the Ax-Kochen transfer principle, \ref{cor:QpZp} is well-known if we replace `for all' by `for almost all'.
The implication $(1)\Longrightarrow(2)$ can be derived explicitly by combining this with the results from \cite{AnscombeKoenigsmann}.
We do, however, not know of any other proof of the implication $(2)\Longrightarrow(1)$. 
We remark that one could 
use the explicit proof of  $(1)\Longrightarrow(2)$ to
construct the formula in $(3)$ to satisfy in addition that
$\phi(\mathbb{Q}_{p})=\mathbb{Z}_{p}$ if and only if $\phi(\mathbb{F}_{p}((t)))=\mathbb{F}_{p}[[t]]$.
We will explore this and further questions of uniform definability of valuation rings in local fields
in a forthcoming paper. 
\end{Remark}

\section{Examples and applications}
\label{section:examples}

\subsection{Diophantine valuation rings: examples and counterexamples}

We now explore $\mathbb{Z}$-fields with and without embedded residue.
We begin by collecting all the known examples of fields without embedded residue,
which gives a common generalization of all the results in the literature of fields with nontrivial diophantine henselian valuation rings.

\begin{Proposition}\label{prop:positive.applications}
In each of the following cases, a $\mathbb{Z}$-field $F$ does not have embedded residue:
\begin{enumerate}
\item $F$ is finite,
\item $F$ is PAC and does not contain a separably closed subfield,
\item $F$ is PRC of characteristic zero and does not contain a real closed subfield,
\item $F$ is P$p$C of characteristic zero and does not contain a $p$-adically closed subfield, for some prime number $p$,
\item $F$ is pseudo-classically closed\footnote{See \cite{Pop} for the definition of a PCC field. We note that
the class of PCC fields contains in particular all PRC fields and all P$p$C fields.} of characteristic zero
and does not contain any real closed or $p$-adically closed subfield (for any $p$),
\item $F=\mathbb{Q}$.
\end{enumerate}
\end{Proposition}

\begin{proof}
(1): The class $\{F\}$ is elementary and does not have embedded residue since finite fields do not admit nontrivial valuations.

(2): The residue field $v$ of any nontrivial valuation $v$ on $F_1\equiv F$ is separably closed, see \cite[Corollary 11.5.5]{FriedJarden},
hence cannot be embedded into any $F_2\equiv F$.

(5): Let $w$ be a nontrivial valuation on some $F_1\equiv F$.
By \cite[Corollary 1.5]{FehmLGP}, also $F_1$ is PCC,
i.e.\ it satisfies a local-global principle with respect to a family $\mathcal{F}$ of separably closed, real closed or $p$-adically closed algebraic extensions. Let $F_1^h$ be a henselization of $F_1$ with respect to $w$.
Then \cite[Theorem 2.9]{Pop} implies that $F_1^h\in\mathcal{F}$, so $F_1w=F_1^hw$ is the residue field of a separably, real or $p$-adically closed field with respect to some henselian valuation and is therefore finite, separably closed, real closed or $p$-adically closed itself.
The same then applies to the algebraic part of $F_1w$.
Thus, since $F$ is by assumption of characteristic zero and does not contain a separably closed, real closed or $p$-adically closed subfield,
there is no embedding of $F_1w$ into any $F_2\equiv F$.

(3): If $F$ is PAC, the claim follows from (2).
If $F$ is PRC but not PAC, then $F$ is contained in a real closed field.
As real closed fields have no $p$-adically closed subfields,
the claim then follows from (5).

(4): If $F$ is PAC, the claim follows from (2).
If $F$ is P$p$C but not PAC, then $F$ is contained in a $p$-adically closed field.
As $p$-adically closed fields have no real closed subfields,
the claim then follows from (5).

(6): Let $v$ be a nontrivial valuation on $F_1\equiv\mathbb{Q}$.
By Lagrange's Four Squares Theorem \cite[Theorem 369]{Hardy-Wright79},
the sums of four squares in $F_1$ form the positive cone of an ordering $<_{F_1}$ 
(which is then necessarily the only ordering of $F_1$), 
so that we have $(\mathbb{Q},<)\equiv(F_1,<_{F_1}$).
If $F_1v$ can be embedded into $F_2\equiv \mathbb{Q}$, then $F_1v$ carries an ordering (namely the restriction of $<_{F_2}$),
from which we deduce by the Baer-Krull Theorem \cite[Theorem 2.2.5]{EP} that $\mathfrak{m}_v$ is convex with respect to $<_{F_1}$.
Since $\mathbb{Q}$ is dense in $\mathbb{R}$ and $(F_1,<_{F_1})\equiv(\mathbb{Q},<)$, we have
$$
 (F_1,<_{F_1})\models\;\forall\varepsilon\!>\!0\;\exists a (0\leq 2-a^{2}<\varepsilon).
$$
So, for $\varepsilon\in\mathfrak{m}_{v}$ with $\varepsilon>_{F_1}0$
there exists $a\in F_1$ such that $0\leq_{F_1}2-a^{2}<_{F_1}\varepsilon$.
Since $\mathfrak{m}_{v}$ is convex, $2-a^{2}\in\mathfrak{m}_{v}$,
so applying the residue map gives $(av)^{2}=2$.
Thus $F_1v\subseteq F_2\equiv\mathbb{Q}$ contains a square root of $2$, a contradiction.
\end{proof}

Combining this with \autoref{cor:single.F} we immediately get the following corollary:

\begin{Corollary}\label{cor:positive.applications}
For each field $F$ in the list of \autoref{prop:positive.applications}, there is an $\exists$-$\mathcal{L}_{\rm ring}$-formula which uniformly defines the valuation ring in $\mathcal{H}(F)$.
\end{Corollary}

\begin{Remark}
This corollary was known before in special cases $(K,v)\in\mathcal{H}(F)$ (without the uniformity statement):
For fields $F$ in (1) by \cite[Theorem 2.6]{Fehm} 
(generalizing the earlier special cases $K=\mathbb{Q}_p$ by \cite[p.~303]{Robinson},
$K$ a finite extension of $\mathbb{Q}_p$ by \cite[Theorem 6]{Cluckersetal},
and $K=\mathbb{F}_q((t))$ by \cite[Theorem 1.1]{AnscombeKoenigsmann}),
for fields $F$ in (2) by \cite[Theorem 3.5]{Fehm}
and for fields $F$ in (3) by \cite[Corollary 3.6]{Fehm}.
To the best of our knowledge, it is new for the fields $F$ in (4)-(6).
We note that for certain special cases of (6), for example for $K=\mathbb{Q}((t))$, it follows from more general results in \cite{Prestel}
that the valuation ring is both $\exists\forall$-$\emptyset$-definable and $\forall\exists$-$\emptyset$-definable.

In the forthcoming work \cite{AF3} we extend the method used in the proof of (6) of \autoref{cor:positive.applications} to further study residue fields of valuations on nonstandard number fields. 
In particular, we show that no number field has embedded residue.
\end{Remark}

\begin{Lemma}\label{prp:perfect.hull.embedded.residue}
Let $F$ be a $\mathbb{Z}$-field. 
If the perfect hull $F^{\mathrm{perf}}$ has embedded residue then $F$ has embedded residue.
\end{Lemma}

\begin{proof}
If $F$ does not have embedded residue, then by \ref{cor:single.F}
there is an $\exists$-$\mathcal{L}_{\rm ring}$-formula $\phi(x)$ that defines $\mathcal{O}_{v_t}$ in $K:=F((t))$.
Since $K^{\rm perf}=\bigcup_{n\in\mathbb{N}}K^{p^{-n}}$ and $K^{p^{-n}}\cong K$,
the same formula $\phi(x)$ defines the valuation ring of the unique extension $v_t^{\rm perf}$ of $v_t$ to $K^{\rm perf}$.
Since $K^{\rm perf}v_t^{\rm perf}=F^{\rm perf}$, \ref{cor:single.F} gives that $F^{\rm perf}$ does not have embedded residue.
\end{proof}

\begin{Proposition}\label{prop:negative.applications}
A $\mathbb{Z}$-field $F$ has embedded residue in each of the following cases:
\begin{enumerate}
\item $F$ contains the separable closure of the prime field.
\item $F$ admits a nontrivial henselian valuation.
\item $F$ is separably closed, real closed, or $p$-adically closed.
\item $F$ is a proper purely transcendental extension of some subfield $F_0$.
\end{enumerate}
\end{Proposition}

\begin{proof}
(1) is a restatement of \autoref{lem:sep.cl} in the special case $C=\mathbb{Z}$.

(2): Let $(F,v)$ be a henselian nontrivially valued field.
By \autoref{prp:perfect.hull.embedded.residue}, we may assume that $F$ is perfect.
We can furthermore assume that $(F,v)$ is $\aleph_0$-saturated,
so if $v$ is of mixed characteristic, 
then $v(\mathbb{Z}\setminus\{0\})$ is contained in a proper convex subgroup of $vF$,
hence we can replace $v$ by an equicharacteristic henselian coarsening.
If $\mathcal{X}$ is a transcendence base of $Fv$ over its prime field $\mathbb{F}$,
then the embedding $\mathbb{F}\rightarrow F$ extends to the relative separable closure of $\mathbb{F}(\mathcal{X})$ in $Fv$,
cf.~\cite[Lemma 2.3]{AnscombeFehm}.
Since $F$ is perfect, this extends further to an embedding of $Fv$ into $F$,
hence $F$ has embedded residue.

(3): If $F$ is separably closed, then the claim follows from (1). If $F$ is $p$-adically closed, then the claim follows from (2).
If $F$ is real closed, then any sufficiently saturated elementary extension of $F$ admits a henselian valuation, so the claim follows from (2).

(4): Without loss of generality, $F=F_0(t)$, in which case $Fv_t=F_0\subseteq F$.
\end{proof}

\begin{Corollary}\label{cor:negative.applications}
For each field $F$ in the list of \autoref{prop:negative.applications}, there is no $\exists$-$\mathcal{L}_{\rm ring}$-formula 
which defines the valuation ring in any $(K,v)\in\mathcal{H}_e'(F)$.
\end{Corollary}

\begin{Remark}
This corollary was known at least for certain $(K,v)\in\mathcal{H}_e'(F)$
for fields $F$ in (1) by \cite[Remark 3.8]{Fehm},
and for fields $F$ in (2) and (3) by \cite[Observation A.1]{AnscombeKoenigsmann} and \cite[Proposition 4.6 and Example 5.4]{FehmPrestel}.
\end{Remark}

\subsection{Diophantine valuation ideals: Examples and counterexamples, and Pop's large fields}
\label{sec:large}

In this section we drop our general assumption that all fields are $C$-fields.
Let $F$ be a field.
By an $F$-variety we mean a separated scheme $X$ of finite type over ${\rm Spec}(F)$.
For a field extension $E$ of $F$ we denote by 
$$
 X(E)={\rm Hom}_{{\rm Spec}(F)}({\rm Spec}(E),X)
$$ 
the set of $E$-rational points of $X$.
Each $x\in X(E)$ gives a scheme theoretic point $\tilde{x}=x(\mathfrak{p})\in X$, where $\mathfrak{p}$ is the unique prime ideal of $E$,
and when we speak of closure or denseness of a subset $A\subseteq X(E)$ in $X$ 
we in fact mean the closure or denseness of the corresponding set $\tilde{A}=\{\tilde{x}:x\in A\}\subseteq X$.
Recall that $F$ is {\em large} (in the sense of Pop)
if it satisfies any of the following equivalent conditions, cf.~\cite[Proposition 1.1]{Pop96}:

\begin{Proposition}\label{fact:Pop.large}
Let $F$ be a field.
Then the following are equivalent:
\begin{enumerate}
\item Every integral $F$-curve with a smooth $F$-rational point has infinitely many such points.
\item For every smooth integral $F$-variety $X$, the set $X(F)$ is empty or dense in $X$.
\item $F$ is existentially closed in the henselization $F(t)^{h}$. 
\item $F$ is existentially closed in the Laurent series field $F((t))$.
\end{enumerate}
\end{Proposition}

For further information on large fields and their relevance in various areas see for example \cite{Jarden, BF, Popsurvey}.
As we will see in \autoref{cor:large}, $F$ is a large field if and only if $F$ is a large $F$-field in the sense of \autoref{def:embedded.residue.large} (where we view $F$ as an $F$-field via the identity map).
Recall that for a $C$-field $F$ we denote by $C_F$ the quotient field of the image of $C\rightarrow F$.
We have to restrict to cases in which sufficiently general versions of resolution of singularities or its local form, local uniformization, have been proven, and here we do not strive for maximal generality:

\begin{Lemma}\label{lem:localuniformization}
Let $F$ be a $C$-field and let $E\subseteq F(t)^h$ be a finitely generated extension of $C_F$ not contained in $F$. 
Then there exists a finite extension $E'$ of $E$ which is the function field of a smooth integral $C_F$-variety $X$ with $X(F)\neq\emptyset$
if one of the following conditions holds:
{\it \begin{enumerate} 
      \item[(a)] $C$ is integral over its prime ring
      \item[(b)] $C$ is a perfect field and $F$ is perfect
      \item[(c)] $F=C_F$
     \end{enumerate}
}
\end{Lemma}

\begin{proof}
We start with case (c).
In this case, $F(t)^h$ is regular and of transcendence degree $1$ over $F=C_F$, hence $E$ is the function field of a geometrically integral $C_F$-curve $X_0$, which always has a smooth projective model $X$.
The restriction of the $t$-adic valuation to $E$ corresponds to an $F$-rational point on $X$.

In cases (a) and (b), $C_F$ is perfect
and the restriction of the $t$-adic valuation to $E$ gives a place with residue field $F_0$ contained in $F$.
By Temkin's inseparable local uniformization \cite[Theorem 1.3.2]{Temkin} there exists a finite purely inseparable extension $E'$ of $E$ 
such that $E'$ is the function field of a smooth integral $C_F$-variety $X$ with an $F_0'$-rational point, where $F_0'$ is a finite purely inseparable extension of $F_0$.
If (a) holds and ${\rm char}(F)=p>0$, then $X$ is defined over some finite field $\mathbb{F}_q$, hence the $q$-Frobenius fixes $X$;
but $(F_0')^{q^k}\subseteq F$ for sufficiently large $k$, hence $X(F)\neq\emptyset$.
If ${\rm char}(F)=0$ or (b) holds, then $F_0'\subseteq F$, so $X(F)\neq\emptyset$.
\end{proof}

\begin{Proposition}\label{prp:large.1}
Let $F$ be a $C$-field 
and suppose that one of the conditions (a), (b) or (c) of \ref{lem:localuniformization} holds.
Then the following are equivalent:
\begin{enumerate}
\item $F$ is a large $C$-field (see \autoref{def:embedded.residue.large}).
\item For every smooth integral $C_F$-variety $X$, the set $X(F)$ is empty or dense in $X$.
\item $F$ and $F(t)^{h}$ have the same $\exists$-$\mathcal{L}_{\rm ring}(C)$-theory.
\item $F$ and $F((t))$ have the same $\exists$-$\mathcal{L}_{\rm ring}(C)$-theory.
\item There is a $C$-embedding $F(t)^{h}\longrightarrow F^{*}$, for some elementary extension $F\preceq F^*$.
\item There is a $C$-embedding $F((t))\longrightarrow F^{*}$, for some elementary extension $F\preceq F^*$.
\end{enumerate}
\end{Proposition}

\begin{proof}
$(3)\Longleftrightarrow (4)$: This follows from $F(t)^{h}\preceq_{\exists}F((t))$, cf.~\cite[Lemma 4.5]{AnscombeFehm}.

$(3)\Longleftrightarrow (5)$ and $(4)\Longleftrightarrow (6)$: These are simple compactness arguments.

$(5)\Longrightarrow(1)$: This is immediate from the definition \autoref{def:embedded.residue.large},
as $E=F(t)^h$ admits the henselian valuation $v_t$ with residue field $F$.

$(1)\Longrightarrow(5)$:
Suppose that $F$ is a large $C$-field,
i.e.~there exist $C$-fields $F_{1},F_{2}\equiv_{C}F$, a $C$-subfield $E\subseteq F_{1}$, and a nontrivial henselian $C$-valuation $v$ on $E$ with residue field $F_{2}$.
By passing to elementary extensions if necessary, we may assume that $E$, $F_{1}$ and $F_{2}$ are $|F|^{+}$-saturated.
In particular we may assume that both $F_1$ and $F_{2}$ are elementary extensions of the $C$-field $F$,
in particular $C_{F_1}=C_{F_2}=C_F$.
Note that $v$ is trivial on $C_{F_1}$ (cf.~\autoref{rem:star}), 
so the residue map is the identity on $C_{F_1}=C_{F}$.
Since $F/C_{F}$ is separable in each of the cases (a)-(c), 
and $E$ is $|F|^{+}$-saturated, 
the identity $C_F\rightarrow C_{F_1}$ extends to a partial section $f:F\longrightarrow E$ of the residue map, cf.~\cite[Lemma 2.3]{AnscombeFehm};
this means in particular that $f$ is an $\mathcal{L}_{\rm vf}(C)$-homomorphism $(F,v_0)\rightarrow (E,v)$, where $v_0$ is the trivial valuation.
Since $v$ is nontrivial on $E$, there exists $s\in\mathfrak{m}_{v}\setminus\{0\}$,
and $f$ extends to a homomorphism $f:(F(t),v_t)\rightarrow(E,v)$ by $f(t):=s$.
Since $v$ is henselian, 
this extends further to a homomorphism $F(t)^h\rightarrow E\subseteq F_1$.

$(2)\Longrightarrow(3)$:
Suppose that $(2)$ holds and let $\varphi$ be an $\exists$-$\mathcal{L}_{\rm ring}(C)$-sentence.
Without loss of generality, $\varphi$ is of the form $(\exists\mathbf{x})\bigwedge_{i=1}^r f_i(\mathbf{x})=0$ with $f_1,\dots,f_r\in C[X_1,\dots,X_n]$,
so there is some closed subset $X_0\subseteq\mathbb{A}^n_{C_F}$ so that $\varphi$ holds in an extension $E$ of $C_F$ if and only if $X_0(E)\neq\emptyset$.
Trivially, if $X_0(F)\neq\emptyset$, then $X_0(F(t)^h)\neq\emptyset$.
Conversely, let $x\in X_0(F(t)^h)$.
Then $\tilde{x}\in X_0$ is the generic point of an integral $C_F$-variety $X_1\subseteq X_0$.
If $x$ is $F$-rational, i.e.~$x$ factors through ${\rm Spec}(F(t)^h)\rightarrow{\rm Spec}(F)$,
then $X_0(F)\neq\emptyset$ and we are done.
Otherwise, the residue field $E:=\kappa(\tilde{x})\hookrightarrow F(t)^h$ of $x$ (which is the function field of $X_1$) 
satisfies the assumptions of \ref{lem:localuniformization},
so there exists a smooth integral $C_F$-variety $X$ with $X(F)\neq\emptyset$ and 
a dominant rational map $X\dashrightarrow X_1$,
i.e.~a non-empty open subvariety $X'$ of $X$ and a morphism $\pi:X'\rightarrow X_1$.
By $(2)$, $X(F)$ is dense in $X$, hence $X'(F)\neq\emptyset$, and thus $X_0(F)\supseteq X_1(F)\supseteq\pi(X'(F))\neq\emptyset$.

$(4)\Longrightarrow(2)$:
Suppose that $(4)$ holds, 
let $X$ be a smooth integral $C_F$-variety with $x\in X(F)$,
and let $X_0\subseteq X$ be a non-empty open subvariety. 
Without loss of generality assume that $X_0$ is affine,
so there exists an $\exists$-$\mathcal{L}_{\rm ring}(C)$-sentence $\varphi$ such that $\varphi$ holds in an extension $E$ of $C_F$
if and only if $X_0(E)\neq\emptyset$.
The base change $X_F:=X\times_{{\rm Spec}(C_F)}{\rm Spec}(F)$ is a smooth $F$-variety,
and $x$ induces an $F$-rational point $x_F$ on $X_F$, which lies on some irreducible component $Y$ of $X_F$.
Note that $Y\cap X_0$ is a non-empty open subset of $Y$.
The local ring $\mathcal{O}_{Y,\tilde{x}_F}$ is contained in a discrete valuation ring $\mathcal{O}$ of the function field $F(Y)$ with residue field $F$, cf.~\cite[Lemma A.1]{JardenRoquette}.
The completion of $\mathcal{O}$ is then $F$-isomorphic to $F[[t]]$.
Since $Y(F(Y))$ is dense in $Y$, so is $Y(F((t)))$, and therefore $(Y\cap X_0)(F((t)))\neq\emptyset$, in particular $F((t))\models\varphi$.
Thus, $(4)$ implies that also $F\models\varphi$, i.e.~$X_0(F)\neq\emptyset$.
\end{proof}

\begin{Corollary}\label{cor:large}
Let $F$ be a field. 
Then $F$ is a large field if and only if $F$ is a large $F$-field.
In particular, if $F$ is a $C$-field which is large as a field, then $F$ is a large $C$-field.
\end{Corollary}

\begin{proof}
The first claim follows by comparing \autoref{fact:Pop.large}(4) and \autoref{prp:large.1}(4) for $C=F$, in which case (c) is satisfied.
The second claim is a direct consequence of this, since if $F$ is large as an $F$-field, then trivially also as a $C$-field.
\end{proof}

\begin{Corollary}
Let $F$ be a field. Then there exists a $\exists$-$\mathcal{L}_{\rm ring}(F)$-formula that defines $\mathfrak{m}_{v_t}$ in $F((t))$ if and only if $F$ is not a large field.
\end{Corollary}

\begin{proof}
If we assume that $F$ is perfect, then this follows immediately from \autoref{cor:large} and \autoref{cor:single.F}.
Without this assumption, we can argue as follows: 
If $F$ is large, then $F$ is existentially closed in $F((t))$ by \autoref{fact:Pop.large}(4),
so if $\varphi(x)$ is an $\exists$-$\mathcal{L}_{\rm ring}(F)$-formula that defines $\mathfrak{m}_{v_t}$ in $F((t))$,
then $F((t))\models \exists x(x\neq 0\wedge\varphi(x))$
implies that there exists $0\neq x\in F\cap \mathfrak{m}_{v_t}$, a contradiction.
If $F$ is not large, then $F$ is not a large $F$-field by \autoref{cor:large},
hence the valuation ideal is uniformly $\exists$-$F$-definable in $\mathcal{H}_e(F)$ by \autoref{thm:universal.characterisation},
so in particular in $(F((t)),v_t)\in\mathcal{H}_e(F)$.
\end{proof}

\begin{Remark}
We note that in cases (a) and (b), $F(t)^h$ and $F((t))$ have the same $\exists$-$\mathcal{L}_{\rm ring}(C)$-theory
as any $(K,v)\in\mathcal{H}_e'(F)$, by \autoref{cor:transfer}.
Therefore, in this case, (3) and (4) of \autoref{prp:large.1}
allow several further equivalent formulations.
\end{Remark}

\begin{Remark}
We do not know whether the statements in \ref{prp:large.1} are also equivalent to
the following statement analogous to (1) of \ref{fact:Pop.large}:
\begin{enumerate}
\item[$(1')$] Every $C_F$-curve with a smooth $F$-rational point has infinitely many such points.
\end{enumerate}
We do not intend to give a comprehensive study of large $C$-fields here,
but we do want to give one sufficient condition for a $C$-field to be large,
which also serves as an illustration of the differences between large fields and large $C$-fields.
\end{Remark}

\begin{Lemma}\label{lem:large.is.E-C.closed}
Let $E\subseteq F$ be an extension of $C$-fields with the same existential $\mathcal{L}_{\rm ring}(C)$-theory.
Then $E$ is a large $C$-field if and only if $F$ is a large $C$-field.
\end{Lemma}

\begin{proof}
Note that $C_E=C_F$ and let $X$ be a smooth integral $C_F$-variety.
For any open subvariety $X_0$ of $X$ (including $X_0=X$),
there is an existential $\mathcal{L}_{\rm ring}(C)$-sentence $\varphi_{X_0}$ such that
for an extension $F'$ of $C_F$, $F'\models\varphi_{X_0}$ if and only if $X_0(F')\neq\emptyset$.
>From this the claim follows immediately.
\end{proof}

\begin{Proposition}\label{lem:PAC.large}
Let $F/E$ be a separable extension of $C$-fields and suppose that $E$ is PAC. Then $F$ is a large $C$-field.
\end{Proposition}

\begin{proof}
Let $\tilde{E}$ denote the relative algebraic closure of $E$ in $F$.
Since $E$ is PAC, so is $\tilde{E}$, and since $F/E$ is separable, $F/\tilde{E}$ is regular.
Thus, $\tilde{E}\preceq_{\exists}F$, see \cite[11.3.5]{FriedJarden},
so in particular $\tilde{E}$ and $F$ have the same existential $\mathcal{L}_{\rm ring}(C)$-theory.
Since PAC fields are large, $\tilde{E}$ is a large $C$-field (\autoref{cor:large}).
The conclusion now follows from \autoref{lem:large.is.E-C.closed}.
\end{proof}

We continue by collecting some examples of $\mathbb{Z}$-fields that are or are not large.

\begin{Proposition}\label{prop:positive.large}
In each of the following cases, a $\mathbb{Z}$-field $F$ is not large:
\begin{enumerate}
\item $F$ is finite.
\item $F=\mathbb{Q}$ 
\item $F$ is finitely generated over its prime field.
\end{enumerate}
\end{Proposition}

\begin{proof}
We prove (3), of which (1) and (2) are special cases.
Let $\mathbb{F}$ denote the prime field of $F$, and $F_0$ the relative algebraic closure of $\mathbb{F}$ in $F$,
and choose a tower of fields $F_0\subseteq\dots\subseteq F_n=F$ such that $F_i=F_{i-1}(X_i)$ 
with a smooth projective geometrically integral $F_{i-1}$-curve $X_i$, for $i=1,\dots,n$.
Choose a smooth projective geometrically integral $\mathbb{F}$-curve $X$ with $X(\mathbb{F})\neq\emptyset$ of genus $g_X>1$ and $g_X>g_{X_i}$ for $i=1,\dots,n$.
Then
$X(F)\supseteq X(\mathbb{F})\neq\emptyset$,
and $X(F)=X(F_n)=\dots=X(F_0)$ by the Riemann-Hurwitz theorem (cf.~\cite[Lemma 6.1.3]{Jarden}).
This latter set of points is finite and not dense, since either $F_0$ is finite,
or $F_0$ is a number field, in which case this follows from Falting's theorem.
Since $\mathbb{Z}_F=\mathbb{F}$, \autoref{prp:large.1}(2) shows that the $\mathbb{Z}$-field $F$ is not large.
\end{proof}

\begin{Corollary}\label{cor:positive.large}
For each field $F$ in the list of \autoref{prop:positive.large}, there is an $\exists$-$\mathcal{L}_{\rm ring}$-formula 
which uniformly defines the valuation ideal in $\mathcal{H}(F)$.
\end{Corollary}

\begin{Remark}
Again this corollary was known before 
for certain fields in (1):
For $K$ a finite extension of $\mathbb{Q}_p$ by \cite[Theorem 6]{Cluckersetal},
and for $K=\mathbb{F}_q((t))$ by \cite[proof of Proposition 3.3]{AnscombeKoenigsmann}).
\end{Remark}

\begin{Proposition}\label{prop:negative.applications.large}
A $\mathbb{Z}$-field $F$ is large in each of the following cases:
\begin{enumerate}
\item $F$ contains the separable closure of the prime field.
\item $F$ admits a nontrivial henselian valuation.
\item $F$ is separably closed, real closed, or $p$-adically closed.
\item $F$ is pseudo-classically closed.
\end{enumerate}
\end{Proposition}

\begin{proof}
(1) is a restatement of \autoref{lem:sep.cl} in the special case $C=\mathbb{Z}$.
(2) and (4) follow from \autoref{cor:large} since both henselian fields 
and pseudo-classically closed fields 
are large fields, see e.g.~\cite[Example 5.6.2, Example 5.6.4]{Jarden}.
(3) is a special case of (4).
\end{proof}

\begin{Corollary}\label{cor:negative.applications.large}
For each field $F$ in the list of \autoref{prop:negative.applications.large}, 
there is no $\exists$-$\mathcal{L}_{\rm ring}$-formula which defines the valuation ideal in any $(K,v)\in \mathcal{H}_e'(F)$.
\end{Corollary}

\subsection{Application: Diophantine henselian valuation rings and valuation ideals on a given field}

In this final subsection
we combine our main result with some of the examples that we just collected to 
show that 
any given field admits at most one nontrivial equicharacteristic henselian valuation 
for which the valuation ring or the valuation ideal is diophantine.
For this we use the notion of the {\em canonical henselian valuation} as defined in \cite[\S4.4]{EP}: 
If $H_1(K)$ resp.~$H_2(K)$ denotes the set of henselian valuations on $K$ with residue field not separably closed
resp.~separably closed, then the canonical henselian valuation $v_K$ is the unique coarsest valuation in $H_2(K)$, if $H_2(K)\neq\emptyset$,
and otherwise the unique finest valuation in $H_1(K)$.

\begin{Theorem}\label{thm:classification}
Let $K$ be a field
and $v$ a nontrivial equicharacteristic henselian valuation on $K$.
If $\mathcal{O}_v$ or $\mathfrak{m}_v$ are $\exists$-$\emptyset$-definable,
then
\begin{enumerate}
\item $v$ is the canonical henselian valuation on $K$,
\item $v\in H_{1}(K)$ and $H_{2}(K)=\emptyset$, and
\item no other nontrivial henselian valuation on $K$ has $\exists$-$\emptyset$-definable 
valuation ring or valuation ideal.
\end{enumerate}
\end{Theorem}

\begin{proof}
We treat the $\exists$-case and the $\forall$-case simultaneously.
By \autoref{cor:single.F}, the $\mathbb{Z}$-field $Kv$ does not have embedded residue resp.~is not large, 
so \autoref{prop:negative.applications}(3) resp.~\autoref{prop:negative.applications.large}(3) shows that $Kv$ cannot be separably closed.
Thus, $v\in H_{1}(K)$.

Furthermore, by \autoref{prop:negative.applications}(2) resp.~\autoref{prop:negative.applications.large}(2), 
$Kv$ cannot admit a nontrivial henselian valuation, so $v$ does not admit any proper henselian refinements.
Thus $v$ is the canonical henselian valuation, so $v_K=v\in H_1(K)$, which implies by definition that $H_{2}(K)=\emptyset$.
Moreover, all other henselian valuations on $K$ are coarsenings of $v$,
in particular also equicharacteristic,
so the argument applies to them as well.
\end{proof}

\begin{Remark}
We remark that all four possibilities of $\mathcal{O}_{v_K}$ and $\mathfrak{m}_{v_K}$ being diophantine or not diophantine can occur.
The following examples of fields $K=F((t))$ with $v_K=v_t$ demonstrate this.
\begin{center}
\begin{tabular}{|c|c|c|c|}
\hline
$F$ & $K$ &  $\mathcal{O}_{v_K}$ diophantine & $\mathfrak{m}_{v_K}$ diophantine\\\hline
$\mathbb{Q}$ & $\mathbb{Q}((t))$ & yes (\autoref{cor:positive.applications}(6)) & yes (\autoref{cor:positive.large}(2)) \\\hline
$\mathbb{Q}(x)$ & $\mathbb{Q}(x)((t))$ & no (\autoref{cor:negative.applications}(4)) & yes (\autoref{cor:positive.large}(3))\\\hline
$\mathbb{Q}^{\rm tr}$ & $\mathbb{Q}^{\rm tr}((t))$ & yes (\autoref{cor:positive.applications}(3)) & no (\autoref{cor:negative.applications.large}(4))\\\hline
$\mathbb{C}$ & $\mathbb{C}((t))$ & no (\autoref{cor:negative.applications}(3)) & no (\autoref{cor:negative.applications.large}(3))\\\hline
\end{tabular}
\end{center}
Here, $\mathbb{Q}^{\rm tr}$ is the field of totally real algebraic numbers, which is a PRC field.
\end{Remark}

We end this work with a discussion of the equicharacteristic assumption in \autoref{thm:classification},
for which we restrict our attention to the valuation ring:
If a $\mathbb{Z}$-field $F$ does not have embedded residue, then the valuation ring is uniformly $\exists$-$\emptyset$-definable in $\mathcal{H}(F)$.
In particular, the valuation ring is uniformly $\exists$-$\emptyset$-definable in the mixed-characteristic henselian valued fields $(K,v)$ with $Kv\equiv F$.
However, in mixed characteristic $(0,p)$ the converse does not hold and there are other reasons why a valuation ring might be $\exists$-$\emptyset$-definable:
For example, if $v(p)$ is minimal positive, then Julia Robinson's formula mentioned in the introduction defines the valuation ring,
and this can easily be extended to valuations with so-called finite initial ramification.
The following example shows that
a field can indeed admit more than one $\exists$-$\emptyset$-definable nontrivial mixed characteristic henselian valuation:

\begin{Example}
Let $F=\mathbb{F}_p((t))^{\rm perf}$ and let $K$ be the field of fractions of the Witt vectors over $F$, see e.g.~\cite[Chapter II \S5]{Serre}.
Then $K$ carries a discrete henselian valuation $u$ with uniformizer $p$ and residue field $F$,
and $\mathcal{O}_u$ is $\exists$-$\emptyset$-definable using Julia Robinson's formula.
But $v_t\circ u$ is a henselian valuation on $K$ with residue field $\mathbb{F}_p$, 
hence $\mathcal{O}_{v_t\circ u}$ is $\exists$-$\emptyset$-definable by \ref{cor:positive.applications}(1).
Thus, all {\em three} henselian valuation rings on $K$ (including the trivial one) are diophantine.
\end{Example}

Since all $\exists$-$\emptyset$-definitions of nontrivial henselian valuation rings in the literature
exploit either properties of the residue field or finite initial ramification,
it seems natural to suspect that non-embedded residue and finite initial ramification 
are the only two reasons why such a valuation ring can be diophantine.
Since at most one of the henselian valuations on a field
can have residue field without embedded residue, and at most one of them can have finite initial ramification,
we would like to pose the following question:

\begin{Question}
Do all fields admit {\em at most three} diophantine henselian valuation rings?
\end{Question}

\section*{Acknowledgements}

\noindent
The authors would like to thank
Hans Schoutens for suggesting Lemma \ref{lemma:Schoutens},
and Florian Pop for suggesting Proposition \ref{prop:positive.applications}(5).
They would also like to thank Immanuel Halupczok, Franziska Jahnke, Jochen Koenigsmann, Dugald Macpherson, Sebastian Petersen, and Alexander Prestel for frequent helpful discussions and encouragement.

\end{document}